\documentclass[10pt,a4paper]{article}

\usepackage{geometry}
\usepackage[utf8]{inputenc}
\usepackage{amsthm, amssymb, natbib, graphicx}
\usepackage{listings}
\usepackage[ruled,vlined]{algorithm2e}
\usepackage{caption}
\usepackage[ruled,vlined]{algorithm2e}
\usepackage{amsmath,tikz}
\usepackage{array, makecell}
\usepackage{sidecap, caption}
\usepackage{xcolor, color} 
\usepackage{todonotes}
\usepackage{comment} 
\usepackage{algorithm2e}
\usepackage[most]{tcolorbox}
\usepackage{soul}
\setcitestyle{numbers,comma,square}
\usepackage{caption}
\usepackage{tabularray}
\usepackage{diagbox}
\UseTblrLibrary{booktabs}

\newtheorem{theorem}{Theorem}[section]

\newtheorem{definition}{Definition}[section]
\newtheorem{remark}{Remark}[section]
\newtheorem{lemma}{Lemma}[section]
\newtheorem{example}[theorem]{Example}

\DeclareMathOperator*{\argmin}{argmin}

\DeclareMathOperator{\spn}{span}
\DeclareMathOperator{\rank}{rank}
\DeclareMathOperator{\diag}{diag}

\newcommand{\ignore}[1]{}
\usepackage{CJKutf8}

\def\x{{\bf x}}
\def\X{{\bf X}}
\def\y{{\bf y}}

\def\w{{\bf w}}

\def\u{{\bf u}}

\def\O{{\cal O}}
\def\M{{\cal M}}
\def\P{{\cal P}}
\newcommand{\R}{\mathbb{R}}
\newcommand{\Z}{\boldsymbol{Z}}
\newcommand{\Q}{\boldsymbol{Q}}
\newcommand{\A}{\boldsymbol{A}}
\newcommand{\q}{\boldsymbol{q}}

\usepackage[normalem]{ulem}

\title{Convergence and Near-optimal Sampling for Multivariate Function Approximations in Irregular Domains via Vandermonde with Arnoldi
}

\author{Wenqi Zhu\thanks{Mathematical Institute, University of Oxford, Oxford, UK, OX2 6GG. 
{\tt wenqi.zhu@maths.ox.ac.uk}; \\$\qquad$ Wenqi Zhu is the corresponding author. This work was supported by the Hong Kong Innovation and Technology Commission (InnoHK Project CIMDA).} \quad and \quad Yuji Nakatsukasa\thanks{Mathematical Institute,  University of Oxford, Oxford, UK, OX2 6GG.  {\tt nakatsukasa@maths.ox.ac.uk} } 
}

\date{\today}

\begin{document}
\maketitle

\abstract{Vandermonde matrices are usually exponentially ill-conditioned and often result in unstable approximations. In this paper, we introduce and analyze the \textit{multivariate Vandermonde with Arnoldi (V+A) method}, which is based on least-squares approximation together with a Stieltjes orthogonalization process, for approximating continuous, multivariate functions on $d$-dimensional irregular domains. The V+A method addresses the ill-conditioning of the Vandermonde approximation by creating a set of discrete orthogonal bases with respect to a discrete measure. The V+A method is simple and general, relying only on the domain's sample points. This paper analyzes the sample complexity of {the least-squares approximation that uses the V+A method}. We show that, for a large class of domains, this approximation gives a well-conditioned and near-optimal $N$-dimensional least-squares approximation using $M=\O(N^2)$ equispaced sample points or $M=\O(N^2\log N)$ random sample points, independently of $d$. We provide a comprehensive analysis of the error estimates and the rate of convergence of the least-squares approximation that uses the V+A method. Based on the multivariate V+A techniques, we propose a new variant of the weighted V+A least-squares algorithm that uses only $M=\O(N\log N)$ sample points to achieve a near-optimal approximation. {Our initial numerical results validate that the V+A least-squares approximation method provides well-conditioned and near-optimal approximations for multivariate functions on (irregular) domains. Additionally, the (weighted) least-squares approximation that uses the V+A method performs competitively with state-of-the-art orthogonalization techniques and can serve as a practical tool for selecting near-optimal distributions of sample points in irregular domains.}}


\section{Introduction and Overview of the Paper}

Many problems in computational science call for the approximation of smooth, multivariate functions. In this paper, we consider the problem of approximating a multivariate continuous function $f: \Omega \rightarrow \R$ of $d \ge 1$ variables, where the domain $\Omega \subset \R^d$ may be irregular. Using Vandermonde matrices to fit polynomials is one of the most straightforward approaches. However, the Vandermonde matrix is usually exponentially ill-conditioned, even on standard domains such as an interval, unless the sample points are very carefully chosen~\cite{vander3, vander2, vander1}. Recently, in \cite{brubeck2021vandermonde}, the authors developed an orthogonalization framework that couples Vandermonde matrices with Arnoldi orthogonalization for univariate function approximations, known as the univariate Vandermonde with Arnoldi method (V+A).

In this paper, we extend the univariate V+A method to a multivariate version that can be used for $d$-dimensional function approximations ($d \ge 2$). 
The multivariate V+A method is based on least-squares approximation coupled with a Stieltjes orthogonalization process, aimed at approximating continuous, multivariate functions on irregular $d$-dimensional domains. This method establishes a set of discrete multivariate orthogonal bases with respect to a discrete measure.

There is extensive literature on $d$-dimensional polynomial approximation algorithms that assume the function $f$ is defined over a hypercube domain containing the irregular domain $\Omega$~\cite{frame2, Adcock2019, cohen}. These algorithms, commonly known as polynomial frame approximations, create an orthogonal basis in the hypercube domain. On the other hand, the V+A algorithm creates a discrete orthogonal basis directly in $\Omega$ and effectively constructs a well-conditioned basis for the irregular domain $\Omega$. 
It is well-known that, even if we have a well-conditioned basis, least-squares approximations can still become inaccurate when the number of sample points $M$ (from a suboptimal distribution, e.g., equispaced points) is insufficient, such as when $M$ is close to the dimension of the approximation space $N$. Poorly distributed sample points can also affect the quality of the solution. In some domains, polynomial frame approximations have provable bounds on the sample complexity; namely, the scaling between the dimension of the approximation space $N$ and the number of samples $M$, which is sufficient to guarantee a well-conditioned and accurate approximation~\cite{ADCOCK, Adcock2019, cohen}. However, to the best of our knowledge, there appears to be no literature on the sample complexity of the V+A procedure.

A key theoretical contribution of this paper is to investigate how $M$ behaves as a function of $N$ such that the least-squares approximant converges to $f$ as the total degree of freedom in approximation $N \rightarrow \infty$. We show that, in a large number of domains (i.e., real intervals, convex domains, or finite unions of convex domains), the least-squares approximation using the V+A method gives a well-conditioned and accurate $N$-dimensional approximation using $M=\O(N^2)$ equispaced sample points or $M=\O(N^2\log N)$ random sample points. The sample complexity of the least-squares approximation using the V+A method is 
comparable to that of 
polynomial frame approximation~\cite{ADCOCK, frame2, cohen}. However, since V+A only requires enough sample points such that the discrete orthogonal polynomials are nearly orthogonal in the whole/continuous domain of interest, it can, in some cases, provide an approximation of similar accuracy using fewer sample points. In addition, since the least-squares problem to be solved becomes well-conditioned, it requires no sophisticated analysis or algorithm to compute an accurate solution (in frame approximation, a remarkable result is that despite the ill-conditioning, one can find a near-optimal solution by finding a small-norm solution using the truncated SVD~\cite{ADCOCK, frame2}).
Using results on sample complexity, we further prove that, under suitable sample distributions and domains, the least-squares approximation that uses the multivariate V+A method is near-optimal. {Specifically, the $N$th polynomial approximation obtained (referred to as the V+A approximant) converges to $f$ at a spectral rate with respect to $N$ for the total degree polynomial space, contingent upon the smoothness of $f$ in $\Omega.$}

In addition to theory, {we provide numerical studies for the multivariate V+A method. Our preliminary numerical results {indicate} that the multivariate V+A method effectively tackles the ill-conditioning associated with the Vandermonde approximation. To our knowledge, the multivariate V+A algorithm introduced in \cite{Hokanson} is implemented only for rational approximations on standard 1D or 2D domains. We extend the numerical applications of the multivariate V+A algorithm to include polynomial approximation on 2D (irregular domains) and higher-dimensional domains ($d=3,5$). Our initial numerical results indicate that as $N$ increases, the V+A least-squares approximation provides a significantly more accurate approximation compared to the Vandermonde least-squares method for $2 \le d\le 5$.} In several papers{~\cite{ADCOCK, adcock2022towards, cohenweight, migdisorth}}, the authors proved that an effective weighting can lead to the near-optimal scaling of $M = \O(N \log(N))$. However, in past approaches~\cite{ADCOCK}, the QR factorization was used to orthogonalize the least-squares system. 
A key numerical contribution of this paper is that we propose a variant of the weighted least-squares algorithm that uses the multivariate V+A as the orthogonalization method (\texttt{VA+Weight}). This algorithm is stable with high probability and only takes $M=\O(N\log N)$ sample points to give a near-optimal approximation. Due to the reduced sample density, \texttt{VA+Weight} also gives a lower online computational cost than the unweighted least-squares using the V+A method. Our initial numerical results confirm that \texttt{VA+Weight} performs competitively with state-of-the-art orthogonalization techniques for the Vandermonde matrix. Finding the optimal distribution of sample points in a high-dimensional irregular domain is an open question in the literature~\cite{Markov}. Our preliminary numerical results highlight that \texttt{VA+Weight} is a practical tool for selecting the near-optimal distribution of sample points {for the irregular domains considered in this paper.} 

The paper is arranged as follows. In Section \ref{Section1.1}, we introduce the Arnoldi orthogonalization procedure in the least squares approximation setup, presenting the univariate V+A algorithm along with three examples of its numerical applications. In Section \ref{section Multivariate VA}, we extend the univariate V+A algorithm to higher dimensions. We compare the multivariate V+A method with other techniques and showcase its effectiveness through five examples of numerical applications in function approximation within $d$-dimensional (irregular) domains ($d \ge 2$).
Section \ref{Section leb constant} gives our main theoretical result on sample complexity and convergence rate of least squares approximation using the V+A algorithm on general $d$-dimensional domains. In Section \ref{section Weighted LSA}, we give the weighted V+A least-squares algorithm, \texttt{VA+Weight}, which takes only $M=\O(N\log N)$ sample points to give a near-optimal approximation.

\subsection{Related Work}

The idea of orthogonalizing monomials and combining Vandermonde with Arnoldi is not completely new. Past emphasis has been on constructing continuous orthogonal polynomials on a continuous domain.  The purpose of the orthogonalization is to obtain the orthogonal polynomial itself \cite{8,9}.  However, V+A, also known as Stieltjes orthogonalization \cite{Stieltjes2, Stieltjes1}, is a versatile method that can dramatically improve the stability of polynomial approximation~\cite{brubeck2021vandermonde}. The purpose of the orthogonalization is to improve the numerical stability of the least-squares system. 

The multivariate version of V+A was first discussed by Hokanson~\cite{Hokanson}{, who} applies V+A to the Sanathanan-Koerner iteration in rational approximation problems. See also~\cite{austin2021practical}. In these studies, the multivariate V+A is used for rational approximations on standard domains. Applying the V+A algorithm to approximate multivariate functions on irregular domains appears not to have been considered in the literature. 

It is worth noting that V+A also has numerous applications beyond polynomial approximation. For instance, it can be used in the `lightning' solver~\cite{Solving}, a state-of-the-art PDE solver that solves Laplace's equation on an irregular domain using rational functions with fixed poles. We also combine V+A with Lawson's algorithm~\cite{Lawson} to improve the accuracy of the approximation. We refer to this algorithm as \texttt{VA+Lawson}. \texttt{VA+Lawson} attempts to find the best polynomial approximation\footnote{The best approximation is defined in Remark \ref{notation} (3).} in difficult domains (for instance, disjoint complex domains). \texttt{VA+Lawson} is also a powerful tool for finding the minimal polynomial for the GMRES algorithms and power iterations in irregular complex domains.

\section{Univariate Vandermonde with Arnoldi}
\label{Section1.1}

Polynomial approximation of an unknown function by fitting a polynomial to a set of sample points from the domain is a classic problem. The interpolation and least-squares methods are two methods to solve this type of problem~\cite{migdisorth, ATAP}. We start with the simplest 1D polynomial approximation problem in this section. 

Let $\Omega \subset \R$ be a bounded domain, $\boldsymbol{X}: = \{\x_i\}_{1 \le i \le M} \subset \Omega$ a set of $M$ distinct sample points and $f:  \Omega \rightarrow \R$ a continuous function which gives a value at each sample point. Let $\P_{1,n}$ represent the space of univariate ($d=1$) polynomials with the highest {order} $n$, and $N:= n+1$ denotes the total degrees of freedom. 
{We aim to find a $N$th degree polynomial approximation, $ \mathcal{L}(f)$, such that
\begin{equation}
    \mathcal{L}(f) = \argmin_{p\in \P_{1,n}} \sum_{i=1}^M  |f(\x_i)-p(\x_i)|^2.
    \label{vander least}
\end{equation}}
We can write $\mathcal{L}(f)(\x) = \sum_{j=1}^N c_j\x^{j-1}$ where $\{c_i\}_{1 \le i \le N}$ are the monomial coefficients to be determined. The equation \eqref{vander least} can be formulated as a Vandermonde least-squares problem
  \begin{equation}
    \label{eq:vander_LS}
\boldsymbol{c} = \argmin_{\boldsymbol{c}\in \R^N} \|\boldsymbol{Ac} - \boldsymbol{\tilde{f}}\|_{2}.
  \end{equation}
Using the pseudoinverse, we can write the solution as $\boldsymbol{c}  =\A^{\dagger}\boldsymbol{\tilde{f}} = (\A^*\A )^{-1} \A^*{\tilde{f}}$ 
where $\A $ is an $M \times N$ Vandermonde matrix with the $(i, j)$th entry $x_i^{j-1}$ for $1\le i \le M$, $1\le j \le N$, $\boldsymbol{c}:=[c_1, \dotsc, c_N]^T$, and $\boldsymbol{\tilde{f}}:=[f({\x}_1), \dotsc, f({\x}_M)]^T$. Since the sample points are distinct, $\A $ is full rank and thus the solution $\boldsymbol{c}$ exists and is unique. If $N=M$, we have an interpolation problem, and $\A $ is a square matrix. The solution is given by $\boldsymbol{c} = \A^{-1}\boldsymbol{\tilde{f}}$. If $N<M$, $\A $ is a tall rectangular matrix and we have a least-squares problem with normal equation $\boldsymbol{A^*Ac} =\boldsymbol{A^*\tilde{f}}$. 

{The least-squares problem~\eqref{eq:vander_LS} has the same solution as the normal equation 
$\boldsymbol{A^*Ac} =\boldsymbol{A^*\tilde{f}}.$ If  $A$ is well-conditioned, the least squares problem can be solved by the normal equation. However,  if  $A$ is moderately (or highly) ill-conditioned, since the condition number of $\mathbf{A^*A}$ is the square of the condition number of $\mathbf{A}$, the conditioning of  the normal equation 
$\boldsymbol{A^*Ac} =\boldsymbol{A^*\tilde{f}}$ significantly worsens.}
In this paper, unless otherwise stated, we focus on the least-squares problem~\eqref{eq:vander_LS} and assume that $N<M$.

Ideally, we can find the coefficients of the polynomial approximation by solving {the least-squares problem \eqref{eq:vander_LS}}. However, the Vandermonde matrices are well known to be exponentially ill-conditioned~\cite{vander2} (unless the nodes are uniformly distributed on the unit circle). 
This creates instability in solving the Vandermonde system and makes a moderately high degree Vandermonde approximation inaccurate. The ill-conditioning of the Vandermonde matrix is due to the non-orthogonal nature of the monomial basis. A potential solution to the ill-conditioning  is to treat the monomial basis as a Krylov subspace sequence, such that 
\begin{equation}
    \spn \{1, \boldsymbol{z}, \boldsymbol{z}^2, \dotsc, \boldsymbol{z}^n \}=\spn\{\q_1, \Z \q_1, \Z ^2\q_1, \dotsc, \Z ^n\q_1 \} = \mathcal{K}_{N}{(\Z , \q_1)}, \label{kylov}
\end{equation}
where $\boldsymbol{z} = [\x_1, \dotsc, \x_M]^T \in \R^{M}$, $\Z  = \diag(\x_1, \dotsc, \x_M) \in \R^{M \times M}$, {$\Z ^j = \diag(\x_1^j, \dotsc, \x_M^j) $} and $\q_1= [1, \dotsc, 1]^T \in \R^{M}$. 
Based on this observation, in V+A we apply Arnoldi orthogonalization to the Krylov space $\mathcal{K}_{N}{(\Z , \q_1)}$. The Krylov space $\mathcal{K}_{N}{(\Z , \q_1)}$ is orthogonalized by the decomposition $\Z \Q_- = \Q \boldsymbol{H}$ where $\Q$ is the  matrix with columns $[\q_1, \dotsc,\q_N]$ and $\Q_-$ is the same matrix
without the final column. By the Arnoldi process, we transform the ill-conditioned Vandermonde system into an optimally conditioned system\footnote{The condition number of $\Q $ is $1$.}, 
\begin{eqnarray}
    \Q^*\boldsymbol{Qd} = \boldsymbol{{Q^*\tilde{f}}}, \quad \text{where}   \quad \Q = \begin{bmatrix} \phi_1({\x}_1) &  \dots   & \phi_N({\x}_1) \\
    \vdots &  \ddots &  \vdots   \\ \phi_1({\x}_M) &  \dots   & \phi_N({\x}_M) \end{bmatrix}  {\quad \text{is orthonormal}\quad \Q^*\Q = M\boldsymbol{I}_N}
    \label{QD=F}
\end{eqnarray}
and $\boldsymbol{\tilde{f}}$  is defined as before. \begin{equation}
\tag{Discrete Orthogonal Polynomials}
    \boldsymbol{\phi}: =\{\phi_1, \phi_2, \dotsc, \phi_N\}
    \label{def discrete orth} 
\end{equation}
is known as a set of \textit{discrete orthogonal polynomials} or a discrete orthogonal basis. $\boldsymbol{d}: = [d_1, \dots, d_N]^{T}$ denotes the coefficient vector related to the discrete orthogonal polynomials, such that  $\mathcal{L}(f)(\x) = \sum_{j=1}^N d_j\phi_j(\x)$. By construction, the discrete orthogonal basis $\boldsymbol{\phi}$ spans the polynomial space $\P_{1,n}$. We say that $\phi_j, \phi_k \in \P_{1,n}$ satisfies \textit{discrete orthogonality} w.r.t. $\boldsymbol{X} = \{\x_i\}_{1 \le i \le M}$, if
\begin{equation}
\tag{Discrete Orthogonality}
\frac{1}{M}\sum_{i=1}^M \phi_j(\x_i) \phi_k(\x_i) =\delta_{j,k}, \qquad 1\le j, k \le N
\label{discrete orth}
\end{equation}
where $\delta_{j,k}$ denotes the Kronecker delta.  On real intervals, the discrete orthogonal polynomials are known to satisfy many algebraic properties analogous to those of the continuous orthogonal polynomials \cite{8,xu2004discrete}. 

\subsection{Algorithm for Vandermonde with Arnoldi}

The Arnoldi algorithm, originally applied {to} finding eigenvalues, uses the {modified} Gram-Schmidt process to produce a sequence of orthogonal vectors. The orthogonal columns  $\{\q_1,\q_2,$ $ \dotsc, \q_N\}$ are obtained by the following recurrence formula
\begin{align}
\boldsymbol{H}_{1,1}:=\frac{\q_1^{*} \Z \q_1}{M}  = \frac{1}{M} \sum_{i=1}^M \x_i, \qquad \boldsymbol{H}_{k+1, k} \q_{k+1}:= \Z  \q_{k} - \sum_{j=1}^k \boldsymbol{H}_{j,k}\q_{j} \quad 
\label{recursion}
\end{align}
where $\q_1= [1, \dotsc, 1]^T \in \R^{M}$ and $\Z  = \diag(\x_1, \dotsc, \x_M) \in \R^{M \times M} $. 
Note that $\q_j^T\q_k =0$ for all $j, k =1, \dotsc, N$ with $j \neq k$ and $\|\q_{k}\|_2 =\sqrt{M}$ for $k =1, \dotsc, N$. The scaling of $\sqrt{M}$ ensures that the Euclidean norm of the V+A solution, $\boldsymbol{d}$, is relatively constant as the number of sample points $M$ increases. 
$\boldsymbol{H}_{j, k}$ are the coefficients of the recurrence formula and also denote the $(j, k)$th entries of the matrix $\boldsymbol{H}$.
 By orthogonality, $\boldsymbol{H}$ is a $N \times (N-1)$ lower Hessenberg matrix. For real sample points, $\{\x_i\}_{1 \le i \le M}$,  $\boldsymbol{H}$ is tridiagonal and the Arnoldi algorithm is equivalent to the Lanczos algorithm.

In the Vandermonde least-squares system \eqref{eq:vander_LS}, we construct {the full matrix $\mathbf{A}$ and solve the poorly conditioned linear system. In V+A, however, we form the matrix $\mathbf{Q}$ column by column, orthogonalizing each new column against all previous columns using the Arnoldi algorithm.} Thus, the Arnoldi process gives us an optimally-conditioned least-squares problem (i.e., {the condition number of $\Q$, $\kappa_2(\Q)=1$}) $\min_{\boldsymbol{d}\in \R^N}\|\boldsymbol{Qd} - \boldsymbol{\tilde{f}}\|_{2}$. The solution for the least-squares system exists, 
\begin{equation}
\boldsymbol{d}= \frac{1}{M} \Q^* \boldsymbol{\tilde{f}}
\label{sol d}
\end{equation}
where the $\frac{1}{M}$ factor comes from the column scaling of $\Q $ such that $\boldsymbol{Q^*Q}=M{\boldsymbol{I}_N}$. Note that the solution of the Vandermonde system, $\boldsymbol{c}$, and the solution of the V+A system, $\boldsymbol{d}$, are related by $\boldsymbol{d} = \frac{1}{M} \Q^{*}\boldsymbol{Ac}$. In MATLAB, $\boldsymbol{d}$ can be obtained either by \eqref{sol d} or by the \texttt{backslash} command.  The \texttt{backslash} command invokes the QR factorization for the least-squares problem and the Gaussian elimination for the interpolation problem (when $M=N$). For improved accuracy, we choose to obtain $\boldsymbol{d}$ using the \texttt{backslash} command for the extra round of orthogonalization~\cite{brubeck2021vandermonde}. 

Once the vector of coefficients $\boldsymbol{d}$ is obtained, the least-squares approximant {$p$} can be evaluated at a different set of points $\boldsymbol{Y}=\{\y_i\}_{1\le i \le K}$. The entries of the $(k+1)$th column of $\boldsymbol{H}$ are the coefficients used in the recursion formula of the discrete orthogonal polynomial, such that 
\begin{equation}
    \boldsymbol{H}_{k+1, k} \phi_{k+1}(\x)= \x \phi_{k}(\x) - \sum_{j=1}^k \boldsymbol{H}_{j,k}\phi_{j}(\x), \qquad  1 \le k \le N-1.
    \label{recursion formula }
\end{equation}
In the polynomial evaluation process, we use the same recursion formula as in \eqref{recursion formula } but apply it on a different set of points, $\boldsymbol{S} =\diag(\y_1, \dotsc, \y_M)$, such that
\begin{equation}
 \boldsymbol{U}_{k+1}:= \frac{1}{\boldsymbol{H}_{k+1, k}} \left( \boldsymbol{S}\boldsymbol{U}_{k} - \sum_{j=1}^k \boldsymbol{H}_{j,k}\boldsymbol{U}_{j} \right),  \quad  1 \le k \le N-1,
\end{equation}
with $\boldsymbol{U}_{1}:=[1, \dotsc, 1]^{T} \in \R^K$ and $\boldsymbol{H}$ is given a priori by the Arnoldi process. The polynomials are evaluated at $\boldsymbol{Y}$ by $
    \boldsymbol{p}:=\boldsymbol{Ud}
$,
where $\boldsymbol{U}:= [\boldsymbol{U} _1, \dotsc, \boldsymbol{U} _N]\in \R^{K\times N}$ and $\boldsymbol{d} \in \R^{N}$ is obtained a priori by the Arnoldi process. The $i$th entry of $\boldsymbol{p}$ represents $\mathcal{L}f(\y_i) = \sum_{j=1}^N d_j \phi_j(\y_i)$ for $1\le i \le K$. Note that the columns of $\boldsymbol{U}$ are in general approximately orthogonal, but not orthogonal. 
To test the validity of the least-squares approximant, we usually evaluate the polynomial approximant on a much finer mesh with $K\gg M$ and compare the evaluated values with $f$. We assess the error of the approximant using this method by computing the maximum absolute value at the evaluation points, $\max_{\mathbf{y}_i \in \mathbf{Y}} |f(\mathbf{y}_i) - \mathcal{L}f(\mathbf{y}_i)|$. 

\textbf{Algorithm and Costs:} The univariate V+A fitting and evaluation is  implemented in \cite{brubeck2021vandermonde} using less than 15 lines of MATLAB code. We made slight variations to the code to improve their efficiency. We provide the algorithm in  Algorithms \ref{V+A Algo} and  \ref{eval Algo}. Instead of using modified Gram-Schmidt (MGS), we use the classical Gram-Schmidt (CGS) {repeated twice, as it gives excellent orthogonality and speed.
Doing CGS twice creates a discrete orthogonal basis that satisfies
$\|\Q^*\Q -\boldsymbol{I}\|_F=\O(MN^{3/2}\boldsymbol{u}) $, where $\boldsymbol{u}$ is the unit roundoff. This is better than MGS, and a good enough bound for polynomial approximation problems which usually have dimensions of $N \lesssim 10^3$ and $M \lesssim 10^6$. More details on the univariate V+A algorithm can be found in \cite{brubeck2021vandermonde}. 

{
\begin{remark} {The notations in the algorithms in this paper represent the following.}
    \begin{itemize}
        \item $[\A]_{:,l}$ refers to all the elements in the $l$th column of the matrix $\A$.
 \item $[\A]_{1:l,l}$ refers to all the elements from the first to the $l$th element in the $l$th row of the matrix $\A$.
 \item $[\A]_{:,1:l}$ refers to all the elements from the first to the $l$th column of the matrix $\A$.
    \end{itemize}
\end{remark} 

\begin{algorithm}[ht]\small
{\textbf{Input}: Sample points $\boldsymbol{X}=\{{\x}_{i}\}_{i=1}^M$ with $\x_i \in \R$, the {order} $n$,  {$\boldsymbol{\tilde{f}} = [f(\x_1), \dotsc, f(\x_M)]^T \in  \R^M$.}
\\ \textbf{Output}: $\Q  \in \R^{M \times N}, \boldsymbol{H}  \in \R^{N \times (N-1)}$, the coefficient of the approximation $\boldsymbol{d} \in \R^{N}$.
\\Initialize $\Q $ and $\boldsymbol{H}$ as zero matrices; Set $[\Q ]_{:,1}$ as a $M \times 1$ matrix of ones. 
\\ \textbf{For} $l = 1, 2, \dotsc, n$
\\ $\qquad$ $\boldsymbol{v} := \diag(\x_1, \dotsc, \x_M)[\Q ]_{:,l};$
\\ $\qquad \quad$\textbf{For} $t = 1, 2$ \text{(Carry out classical Gram-Schmidt twice)} 
\\ $\qquad$ $\qquad$ $\boldsymbol{s}  :=\frac{1}{M} [\Q ]^*_{:,1:l} \boldsymbol{v};\quad $ $\boldsymbol{v}:=\boldsymbol{v} -   [\Q ]_{:,1:l}\boldsymbol{s};$
\\ $\qquad$ $\qquad$ $[\boldsymbol{H}]_{1:l,l} :=[\boldsymbol{H}]_{1:l,l} + \boldsymbol{s};$
\\$\qquad \quad$ \textbf{end} 
\\$\qquad$ $[\boldsymbol{H}]_{{l+1},{l}}:= \frac{1}{\sqrt{M}} \|\boldsymbol{v}\|_2; \quad   [\Q ]_{:,{l+1}}:=\boldsymbol{v}/\|\boldsymbol{v}\|_2$;
\\ \textbf{end} 
\\ $\boldsymbol{d}=  \Q  \backslash \boldsymbol{\tilde{f}}. $ Solve the least-squares problem by MATLAB backslash command. }
\caption{\small Polynomial Fitting using Univariate Vandermonde with Arnoldi (Adapted from~\cite{brubeck2021vandermonde})
\label{V+A Algo}}
\end{algorithm}

\begin{algorithm}[ht]\small
{\textbf{Input} :   Evaluation points $\boldsymbol{Y}=\{{\y}_i\}_{i=1}^{K}$ with $\y_i \in \R$, the order $n$, $\boldsymbol{H}  \in \R^{N \times (N-1)}$, $\boldsymbol{d} \in \R^{N}$; 

\textbf{Output}: ${\boldsymbol{p}} \in \R^{K}$ which is the vector of values at evaluation points $\boldsymbol{Y}$. 

Initialize $\boldsymbol{U}$ as a $K\times N$ zero matrix and set $[\boldsymbol{U}]_{:,1}$ as a $K \times 1$ matrix of ones;
\\ \textbf{For} $l = 1, 2, \dotsc, n$
\\$\qquad$ $\boldsymbol{v}:= \diag(\y_1, \dotsc, \y_K)[\boldsymbol{U}]_{:,l};$
\\$\qquad$ $\boldsymbol{v}:=\boldsymbol{v}- [\boldsymbol{U}]_{:,1:l}[\boldsymbol{H}]_{1:l,l};\quad$ $[\boldsymbol{U}]_{:,l}:=\boldsymbol{v}/[\boldsymbol{H}]_{l+1,l};$
\\ \textbf{end} 
\\ $\boldsymbol{{p}}:=\boldsymbol{Ud}$}
\caption{\small Evaluating Functions for Univariate Approximation~\cite{brubeck2021vandermonde}}
 \label{eval Algo}
\end{algorithm}
}

\subsection{Applications and Numerical Examples for Univariate V+A}
{In the first paper on V+A~\cite{brubeck2021vandermonde}, several applications of the univariate V+A method are provided, including interpolation at Chebyshev points, least-squares approximation on two intervals, Fourier extension, Laplace equation, and conformal mapping.} In addition to these applications presented in  \cite{brubeck2021vandermonde}, we explore three additional examples of V+A least-squares polynomial approximations.

\begin{example} \textbf{(Disjoint Domain)}
 Approximating  $f(x)=x\cos(10x)$ using $M=N^2$ equispaced sample points in a disjoint domain $[-3,-1] \cup [3, 4]$. This example challenges the algorithm's ability to handle a disjoint domain. We compare the V+A method to the Vandermonde method in the left plot of Figure \ref{1d}. Initially, the two approximations give the same error. However, the Vandermonde system has an error stagnating at $10^{-4}$ for $N>30$ due to ill-conditioning of $\A $, while the V+A method gives an error reduction down to $10^{-15}$ as $N$ increases.
\end{example}

\begin{example} 
\textbf{(Non-Smooth Function)} Approximating $f(x) = |x|$ in $[-1, 1]$ using $M=N^2\log N$ random sample points. This example tests the algorithm on approximating a non-smooth function with random uniform sample points. Similar to \textit{Example 1}, V+A also yields a much better approximation than the Vandermonde approximation (the middle plot of Figure \ref{1d}). The error for V+A is also in line with the error of the best polynomial approximation ($\sim \frac{0.28}{n}$)~\cite{stahl1993best}. 
\end{example}

\begin{example} 
\textbf{(Infinite Domain)} Approximating $f(x)=e^x$ in $[-10^3,-10^{-3}]$ using $M=N^2$ logarithmic equispaced points. {The logarithmic equispaced points are defined as $M$ points in the form of $-10^{\Gamma}$, where $\Gamma$ is chosen as equispaced points in the interval $[-3, 3].$} This example focuses on sample points generated by a different measure over a wide interval. The Vandermonde method fails for this problem but the least squares approximation using the V+A algorithm gives a stable error reduction for all $N$ (the right plot of Figure \ref{1d}). Unlike some methods used in~\cite[Sec. 4]{Transplant} which involves transplantation of the domain, the least squares approximation using V+A Algorithm is carried out directly on $[-10^3,-10^{-3}]$. This example illustrates that the V+A algorithm can adapt to different domains and different discrete measures. 
\end{example}

\begin{figure}[!ht]
\centering
\includegraphics[width=15cm]{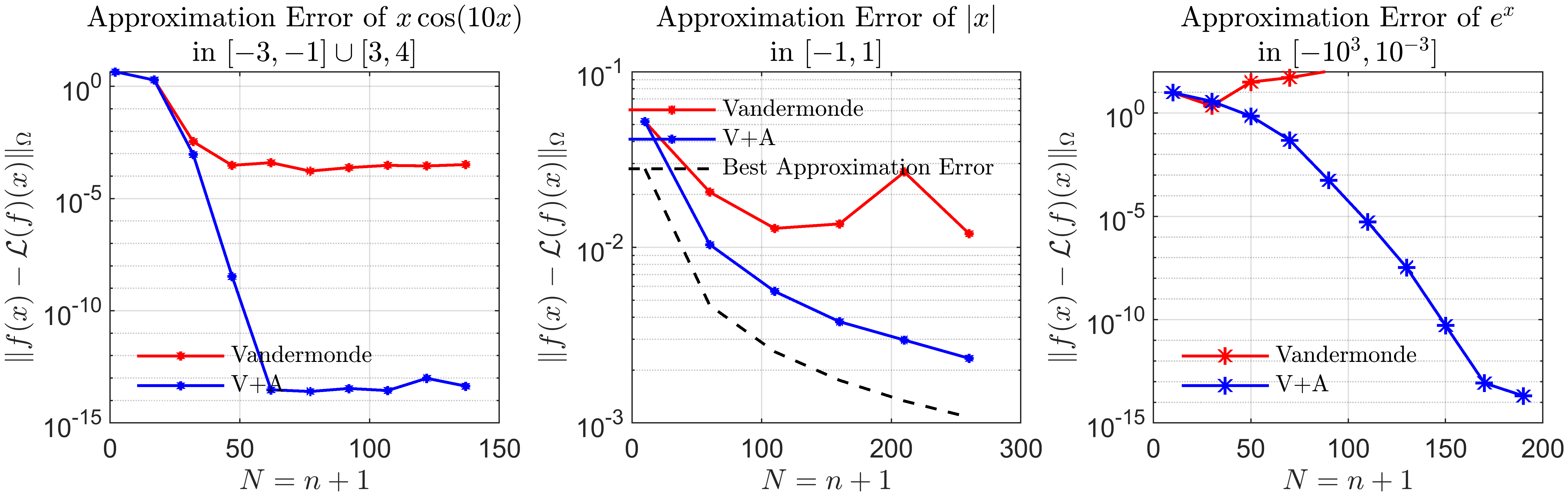}
\caption{\small The least squares approximation with the V+A method is computed using the univariant V+A (Algorithm \ref{V+A Algo}). The Vandermonde approximation is computed using \texttt{polyfit/polyval} provided
in MATLAB. }
\label{1d}
\end{figure}

\section{Multivariate Vandermonde with Arnoldi} 
\label{section Multivariate VA} 

The V+A method can be readily extended to higher dimensions ($d>1$). 
Let $\Omega \in \R^d$ be the domain, and $\x = [x_{(1)}, \dotsc, x_{(d)}]^T \in {\R}^{d}$ be the variables. For instance, for the term $x_{(1)}^{\alpha_1} x_{(2)}^{\alpha_2} \dotsc x_{(d)}^{\alpha_d}$, we have $0\le \alpha_r \le n$ for all $r = 1, \dotsc, d$. 
We   denote the associated multi-index set as $\boldsymbol{\alpha}= [\alpha_1, \alpha_2, \dotsc, \alpha_d]^T \in {[\mathbb{Z}{[0, n]}]}^{d}$ and $n$ as the {order of index} of the polynomial space. Each index $\alpha_r \in \mathbb{Z}{[0, n]}$ is an integer between $0$ to $n$ inclusively.  Let $\P_{d,n}$ denote the general representation of the polynomial spaces where $\x = [x_{(1)}, \dotsc, x_{(d)}]^T \in {\R}^{d}$ and {$n$ is the highest order in the index.} We give the definition of two polynomial spaces as follows:
\begin{eqnarray*}
\text{Maximum Degree $\P^M_{d,n}$ with indices} \quad && \mathcal{I}^P: = \bigg\{\boldsymbol{\alpha} \in {[\mathbb{Z}_{[0, n]}]}^{d}, \max_{1 \le r \le d} \alpha_r\le n\bigg\},
\\
\text{Total Degree $\P^T_{d,n}$ with indices}  \quad && \mathcal{I}^T: = \bigg\{\boldsymbol{\alpha} \in {[\mathbb{Z}_{[0, n]}]}^{d}, \sum_{r=1}^d \alpha_r\le n \bigg\}.
\end{eqnarray*}
In these polynomial spaces, the maximum sum of components of $\boldsymbol{\alpha}$ and the number of basis functions are 
\begin{eqnarray}
\label{a1 1}
\text{Maximum Degree $\P^M_{d,n}$:} \quad && \|\boldsymbol{\alpha}\|_1=nd, \quad N:=|\mathcal{I}^P| = (n+1)^d,
\\
\text{Total Degree $\P^T_{d,n}$:}  \quad && \|\boldsymbol{\alpha}\|_1=n, \quad N:=|\mathcal{I}^T| = \binom{n+d}{n}
\label{a1 2}
\end{eqnarray}
where $N$ denotes the number of basis functions in each polynomial space. We also refer to the total degrees of freedom as $N$.  
In this paper, the theory applies to multivariate functions for all $d \ge 1$. The numerical experiments primarily focus on the case $d=2$. Unless otherwise stated, the numerical examples for $d=2$ are conducted in the total degree polynomial space with $N := \frac{1}{2}(n+1)(n+2)$. Examples for $d > 2$ are presented in Figure \ref{fig highD example} and discussed in Example \ref{highD example}.

\begin{remark}
Although the number of basis functions in the polynomial spaces satisfies $N = \mathcal{O}(n^d)$ as $n \rightarrow \infty$, for a fixed pair $(n, d)$, the number of basis elements, $N$, is larger when we employ the maximum degree compared to the total degree. 
{The total degree space is a subset of the maximum degree space, $\P^T_{d,n}  \subseteq \P^M_{d,n}$.}
\end{remark}

\begin{remark} \textbf{(Norms and Notation)}
For $\mathbf{x} \in \R^d$ and a domain $\Omega \subset \R^d$, we provide the following definitions for norms. 
    \begin{enumerate}
    \item For a finite set of points $\boldsymbol{X}=\{\x_i\}_{1 \le i \le M}$ and bounded functions $f, g: \boldsymbol{X} \rightarrow \R$,  we define\textit{ the infinity $\boldsymbol{X}$-norm} of $g: \boldsymbol{X} \rightarrow \R$ as $\|g\|_{\boldsymbol{X}}=\max_{\x \in \boldsymbol{X}}|g(\x )|$ and $\langle f , g \rangle_M =  \frac{1}{M}\sum_{i=1}^M f(\x_i) g(\x_i)$. 

    \item   For any bounded domain $\Omega$, we define \textit{the infinity $\Omega$-norm} of a bounded function $g: \Omega \rightarrow \R$ as $\|g\|_{\Omega}=\sup_{\x \in\Omega}|g(\x )|$. 
    \item $p^*$ is the best approximation for $f$ in $\P_{d,n}$ if $p^*$ satisfies $\|f(\mathbf{x}) - p^*(\mathbf{x})\|_{\Omega} \le \|f(\mathbf{x}) - p(\mathbf{x})\|_{\Omega}$ for all $p \in \P_{d,n}$. 
    \item $ \|\nabla p(\x)\|_{\Omega,2}: = \sup_{\x \in \Omega} \|\nabla p(\x)\|_2  =  \sup_{\x \in \Omega} \bigg(\sum_{r=1}^d \big|\frac{\partial p}{\partial x_{(r)}}(\x) \big|^2\bigg)^{1/2}$.
    \item   $ \|\nabla p(\x)\|_{\Omega}: = \sup_{\x \in \Omega} \|\nabla p(\x)\|_{\infty}= \sup_{\x \in \Omega}  \bigg(\max_{1 \le r \le d} \big|\frac{\partial p}{\partial x_{(r)}}(\x) \big|\bigg)$. $\|\nabla p(\x)\|_{\Omega} \le \|\nabla p(\x)\|_{\Omega, 2}$. 
    \end{enumerate}
    \label{notation}
\end{remark}

The key distinction between multivariate V+A and univariate V+A is that the multivariate monomial basis does not correspond to a Krylov subspace. Consequently, the multivariate monomial basis lacks a canonical ordering, meaning there is no universal order to list the columns of a multivariate Vandermonde matrix.
Therefore, in the multivariate V+A algorithm, we employ specific basis-ordering strategies. A new column is created by carefully selecting one coordinate from the sample points 
to form a diagonal matrix, which is then multiplied to a particular preceding column. Subsequently, we orthogonalize this new column against previous columns using the Gram-Schmidt process. {Specifically, let the set of multi-indices be \(\mathcal{I}\), and let \(\mathcal{I}[k]\) denote the \(k\)th element of the multi-index set and $\mathcal{I}[1] =\textbf{0} \in \R^n$. For $l= 1,2, \dotsc$, we select \(k\) by finding the smallest \(k\) such that \(\mathcal{I}[l + 1]:=\mathcal{I}[k] + e_r \), where \(e_r\) is the \(r\)th column of the identity matrix. Such orderings are not unique. For instance, for total degree polynomials, a grevlex ordering (which orders terms first by total degree and then lexicographically) satisfies this condition. For a third-degree total degree polynomial with two variables, the ordered basis could be
$
\mathcal{I} = \{(0,0), (1,0), (0,1), (2,0), (1,1), (0,2), (3,0), (2,1), (1,2), (0,3)\}.
$
}
}More details are given in \cite[Sec 2.1]{Hokanson}. The complete algorithm for multivariate V+A is also provided in \cite{Hokanson} and outlined in Algorithm \ref{Multi V+A Algo} for completeness. 

\begin{algorithm}[ht]\small
{\textbf{Input}: Sample points $\boldsymbol{X}=\{\x_{i}\}_{i=1}^M$ with $\x_i\in  \R^d$, total degrees of freedom $N$, the index set $\mathcal{I}$ with length $N$, the function values at $\boldsymbol{X}$ i.e., $\boldsymbol{\tilde{f}} \in  \R^M$.
\\ \textbf{Output}: $\Q  \in \R^{M \times N}, \boldsymbol{H}  \in \R^{N \times N}$, the coefficient $\boldsymbol{d} \in \R^{N}$.
\\Set $\Q $ and $\boldsymbol{H}$ as zero matrices; Set $[\Q ]_{:,1}$ as a $M \times 1$ matrix of ones and  $[R]_{1,1}=1$. 
\\ \textbf{For} $l = 1, 2, \dotsc, |\mathcal{I}|-1$
\\ $\qquad$ Choose the smallest $k$ such that $\exists r$ where     $ \mathcal{I}[l + 1]:=\mathcal{I}[k] + e_r $;
 \\ $\qquad$ $\boldsymbol{v} := \diag({x_{(r)}}_1, \dotsc, {x_{(r)}}_M)[\Q ]_{:,k};$
\\ $\qquad \quad$ \textbf{For} $t = 1, 2$
\\ $\qquad$  $\qquad$  $\boldsymbol{s}  := \frac{1}{M} [\Q ]^*_{:,1:l} \boldsymbol{v};\qquad$ $\boldsymbol{v}:=\boldsymbol{v} -  [\Q ]_{:,1:l}\boldsymbol{s};$
\\ $\qquad$ $\qquad$  $[\boldsymbol{H}]_{1:l,l+1} :=[\boldsymbol{H}]_{1:l,l+1} + \boldsymbol{s};$
\\   $\qquad \quad$ \textbf{end}
\\
$\qquad$ $[\boldsymbol{H}] _{{l+1},{l+1}}:=\frac{1}{\sqrt{M}}\|\boldsymbol{v}\|_2; \quad [\Q ]_{:,{l+1}}:=\boldsymbol{v}/\|\boldsymbol{v}\|_2$;
\\\textbf{end}
\\$\boldsymbol{d}=   \Q  \backslash \boldsymbol{\tilde{f}}. $ Solve the least-squares problem by MATLAB backslash command. }
\caption{\small Polynomial Fitting using Multivariate V+A~\cite{Hokanson}}
\label{Multi V+A Algo}
\end{algorithm}

\begin{algorithm}[ht]\small
{
\textbf{Input} :   Evaluation points $\boldsymbol{Y}=\{\boldsymbol{y}_i\}_{i=1}^{K}$ with each $\boldsymbol{y}_i\in \R^d$, $\mathcal{I}$, $\boldsymbol{H}  \in \R^{N \times N}$, $\boldsymbol{d} \in \R^{N}$; 

\textbf{Output}: ${\boldsymbol{p} } \in \R^{K}$ is the vector of values at evaluation points $\{\boldsymbol{y}_i\}_{i=1}^{K}$. 

Initialize $\boldsymbol{U}$ as $K\times N$ zero matrix and set $[\boldsymbol{U}]_{:,1}$ as a $K \times 1$ matrix of ones;
\\ \textbf{For} $l = 1, 2, \dotsc, |\mathcal{I}|-1$
\\ $\qquad$  Choose the smallest $k$ such that $\exists r$ where    $ \mathcal{I}[l + 1]:=\mathcal{I}[k] + e_r $; 
\\ $\qquad$  $\boldsymbol{v} := \diag({y_{(r)}}_1, \dotsc, {y_{(r)}}_K)[\boldsymbol{U}]_{:,k};$
\\ $\qquad$  $\boldsymbol{v}:=\boldsymbol{v}- [\boldsymbol{U}]_{:,1:l}[\boldsymbol{H}]_{1:l,l+1};\quad$ $[\boldsymbol{U}]_{:,l+1}:=\boldsymbol{v}/[\boldsymbol{H}]_{l+1,l+1};$
\\ \textbf{end}
\\ $\boldsymbol{{p}}:=\boldsymbol{Ud}$}
\caption{\small Evaluating Function for Multivariate Approximation~\cite{Hokanson}}.
 \label{Multi eval Algo}
\end{algorithm}

\subsection{Applications and Numerical Examples for Multivariate V+A}

In \cite{Hokanson}, the multivariate V+A algorithm is applied for rational approximations on standard 1D or 2D domains. In this section, we provide numerical examples demonstrating the application of the multivariate V+A algorithm for polynomial approximation on 2D irregular domains, tensor product domains, and higher-dimensional domains ($d>2$).

\begin{example}
\textbf{(V+A Polynomial Approximation on 2D Tensor Product Domain)} In Figure \ref{2D example}, we plot the results of approximating a smooth function $f(x_{(1)},x_{(2)})=\sin(\frac{x_{(1)}^2+x_{(2)}^2+x_{(1)}x_{(2)}}{5})$ in a tensor-product domain using $M=N^2$ equispaced points\footnote{More details on the sampling complexity are discussed in Section \ref{Deterministic sample points}.}. The Vandermonde's least-squares system quickly becomes highly ill-conditioned at higher degrees, causing the error of the Vandermonde method to stagnate at $\O(1)$. On the other hand, the least squares approximation using the multivariate V+A method gives a stable error reduction\footnote{The approximation errors are measured in an equispaced mesh $\boldsymbol{Y}=\{\boldsymbol{y}_i\}_{1\le i \le K}$ with $K = 3 M$, that is, finer than the sample points.} down to $10^{-15}$ for $N=500$. Also, as shown in the middle plot of Figure \ref{2D example}, the error obtained from the multivariate V+A is small throughout the domain with no spikes near the boundary. This is because the discrete orthogonal basis generated by the Arnoldi orthogonalization is well approximated by the continuous orthogonal basis with $M=\O(N^2)$ points. 

\begin{figure}[!ht]
    \centering
\includegraphics[width=15cm]{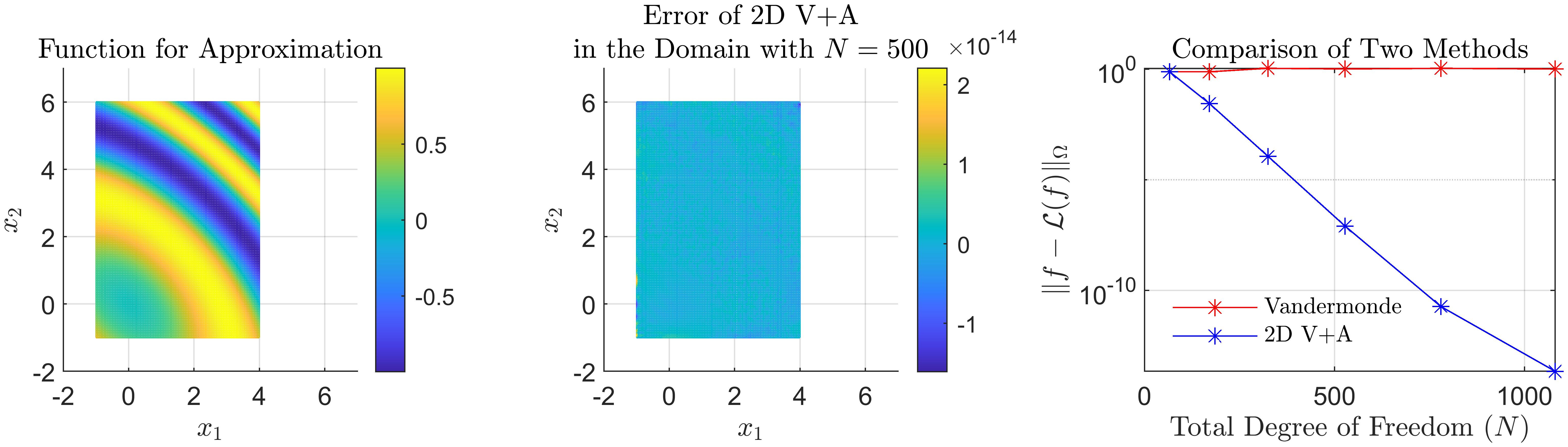}
    \caption{\small Approximating $f(x_{(1)},x_{(2)})=\sin(\frac{x_{(1)}^2+x_{(2)}^2+x_{(1)}x_{(2)}}{5})$ on $[-1, 4] \times [-1, 6]$ using equispaced mesh with $M=N^2$.}
    \label{2D example}
\end{figure}
\end{example}

\begin{example}
\textbf{(V+A Polynomial Approximation on 2D Irregular Domain)} We approximate the same function on an elliptical domain. We sample the points using the rejection sampling method. Namely, we enclose the domain $\Omega$ in a hypercube domain $\Omega_{cube}=[0,4]\times [0,6]$ and draw independent and identically distributed (i.i.d.) random samples from the uniform probability measure on ${\Omega_{cube}}$. 
The acceptance rate for our domain is $\text{Area}(\Omega)/\text{Area}(\Omega_{cube}) \approx 59.59\%$. We end the rejection sampler once we have a total of $M=N^2\log N$ sample points\footnote{More details on the sampling complexity are discussed in Section \ref{Randomized sample points}.}. We plot the bivariate approximation result for the irregular domain in Figure \ref{2D example 2}. Again, the Vandermonde method fails for any $N>100$, but the least squares approximation using the V+A method gives a stable error reduction to $10^{-15}$ for $N=600$. Since the Arnoldi orthogonalization generates a basis in the domain, the error for the least-squares approximation does not worsen when we switch from the tensor-product domain to the irregular domain. This example illustrates that the V+A algorithm can adapt to different discrete measures and non-tensor product domains.

\begin{figure}[!ht]
    \centering
\includegraphics[width=14.5cm]{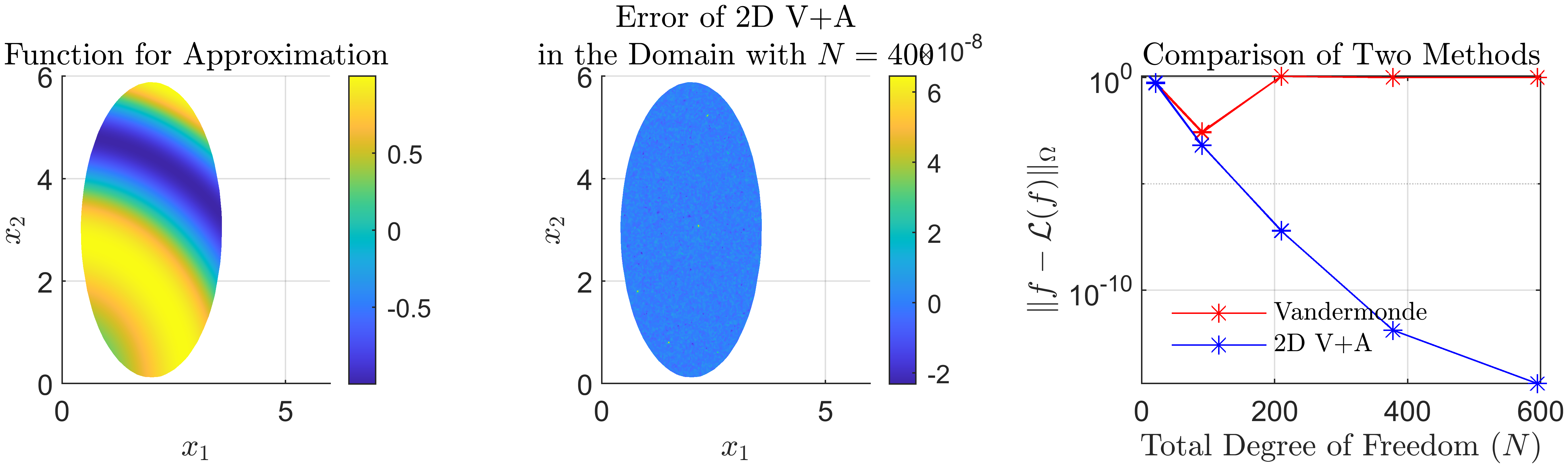}
    \caption{\small Approximating $f(x_{(1)},x_{(2)})=\sin(\frac{x_{(1)}^2+x_{(2)}^2+x_{(1)}x_{(2)}}{5})$ on an elliptical domain ({Domain $4$ in Figure \ref{domain}}) using randomized mesh with $M=N^2\log N$ (See Theorem \ref{Ramdon AM} for more details on sample complexity). Here and throughout this paper, we plot the mean of $25$ random trials when the sample points are random. Unless otherwise stated, we use the total degree polynomial space for approximation in numerical experiments.}
    \label{2D example 2}
\end{figure}
\end{example}

\begin{example}
\textbf{(V+A Polynomial Approximation for 2D Non-Differential Functions)} We test the least squares approximation with the V+A method on a non-differential function $f(x_{(1)},x_{(2)})=|x_{(1)}-2||x_{(2)}-3|$ in the irregular domain. We found that the error reduction for the non-smooth function is much slower than the error reduction for the smooth function as illustrated in the right plot of Figure \ref{2D example 3}. The slow convergence in the error reduction is not caused by the least squares approximation using the V+A method; it is because polynomials are not great for approximating non-smooth functions. In other words, no polynomials in the space $\P^T_{d,n}$ can converge rapidly to such non-differentiable functions. Additionally, the smoothness requirement of the function is stricter for higher dimensional domains. As we observe from the right plot of Figure \ref{2D example 3}, the 2D least-squares approximation of a non-smooth function exhibits a slower error reduction than the 1D least-squares approximation of a non-smooth function. The theoretical results for convergence and explanations of these phenomena are discussed in Section \ref{sec Convergence Rate} and Theorem \ref{thm convergence rate equispaced}. 
\end{example}

\begin{example}
\label{example 3.4}
\textbf{(Orthogonal Basis on $\Omega$ vs. on $\Omega_{\text{cube}}$)}
We compare the least squares approximation using the V+A method with the least squares approximation using the orthogonal basis on the bounding domain (referred to as the `orthogonal basis on the bounding domain approximation'). In the orthogonal basis on the bounding domain approximation, we enclose the irregular domain in a hypercube domain, $\Omega \subset \Omega_{\text{cube}}$, and create an orthonormal basis for the bounding tensor-product domain using orthonormal Legendre polynomials\footnote{{This method is somewhat close to polynomial frame approximation~\cite{Adcock2019}, however it is not the same as we do not use regularization, e.g. truncated SVD, as here we aim to observe the effect of the V+A orthogonalization. More generally, we view an extensive comparison between frame approximation and V+A-based approximation a topic for future work.}}.
We generate a numerical example using the function and the domain shown in Figure \ref{2D example 2}. In the numerical experiments, the 2D orthonormal Legendre polynomial basis is generated by sampling equispaced points on the tensor product domain and using orthogonalization techniques (such as QR or Gram-Schmidt process) to obtain the discrete orthogonal Legendre polynomials. An illustration of the sample points where the orthogonal polynomials are generated is given in Figure \ref{figure illustration of sample points}. 

As given in the left plot of Figure \ref{2D example 3}, both methods provide a well-conditioned and accurate approximation. Using the same number of sample points $M=N^2 \log N$, the orthogonal basis on the bounding domain approximation is slightly less accurate than the least squares approximation with the V+A method. This is partly because, in the orthogonal basis on the bounding domain approximation, the sample points are distributed across the entire tensor-product domain $\Omega_{\text{cube}}$, with only about 59.59\% of the sample points inside the domain $\Omega$. 
In contrast, in the least squares approximation with the V+A method, all sample points are within $\Omega$ and contribute to the approximation. 

Even if we increase the number of sample points for the orthogonal basis on the bounding domain approximation to $M=\frac{1}{0.5959}N^2 \log N$ to have approximately the same number of sample points within $\Omega$, the orthogonal basis on the bounding domain approximation with increased sample points remains slightly less accurate. This is because the Legendre polynomials are orthonormal for the tensor domain rather than for the elliptical domain. The orthonormal basis is no longer orthogonal when restricting the Legendre polynomials to an elliptical subdomain.
\end{example}

\begin{figure}[!ht]
    \centering
\includegraphics[width=7cm, height = 4.5cm]{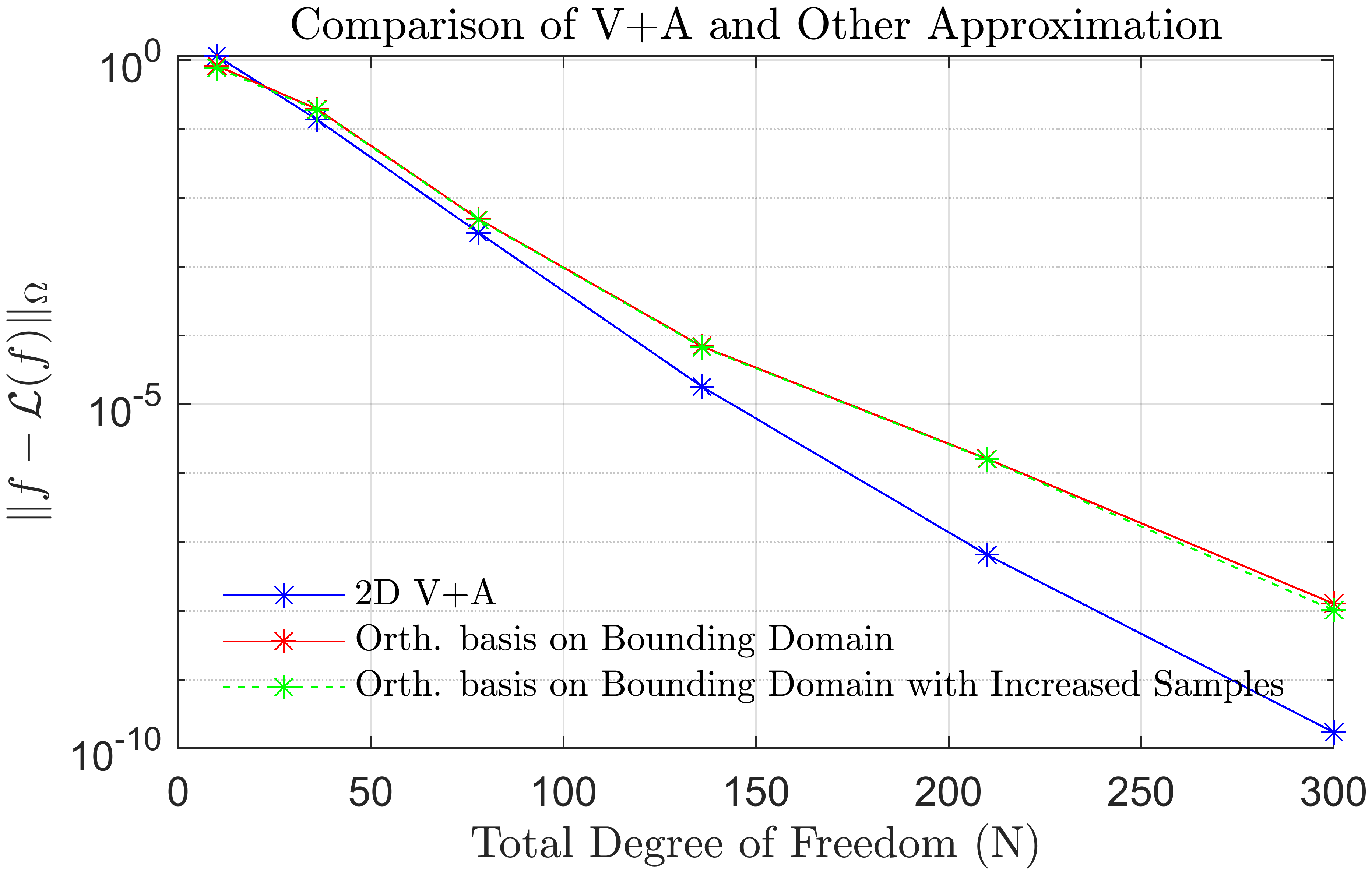}
\includegraphics[width=7cm, height = 4.5cm]{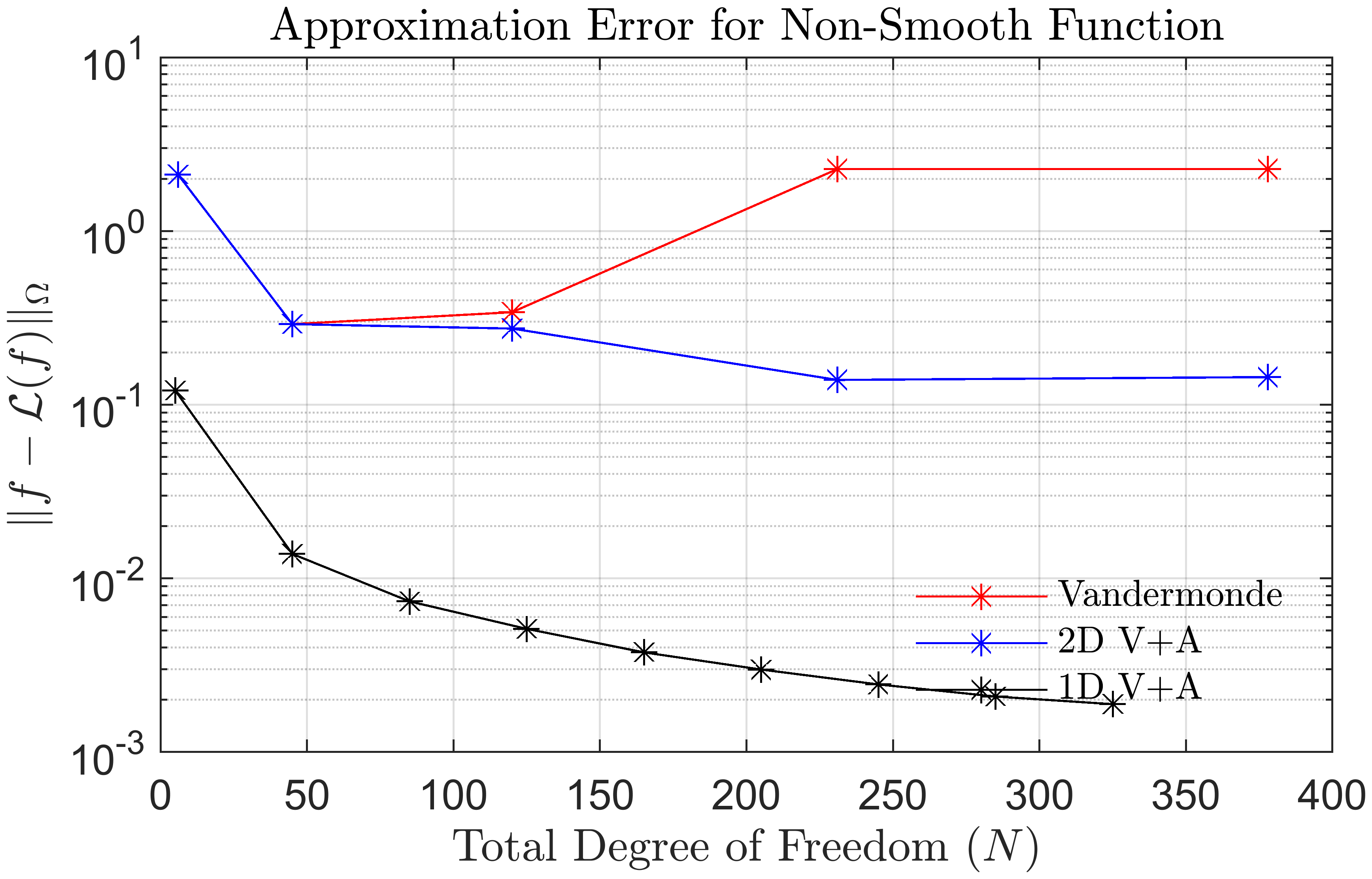}
    \caption{\small \textbf{Left}: Comparison of {the orthogonal basis on the bounding domain approximation and the least squares approximation with the V+A method} for $f(x_{(1)},x_{(2)})=\sin(\frac{x_{(1)}^2+x_{(2)}^2+x_{(1)}x_{(2)}}{5})$ on an elliptical domain. \textbf{Right}: Approximating non-smooth function $f(x_{(1)},x_{(2)})=|x_{(1)}-2||x_{(2)}-3|$ on an elliptical domain and $f(x)=|x-2|$ in $[0,4]$, both {using $\O(N \log N)$ randomized sample points}.}
    \label{2D example 3}
\end{figure}

\begin{figure}[h!]
    \centering
    \includegraphics[width=10cm]{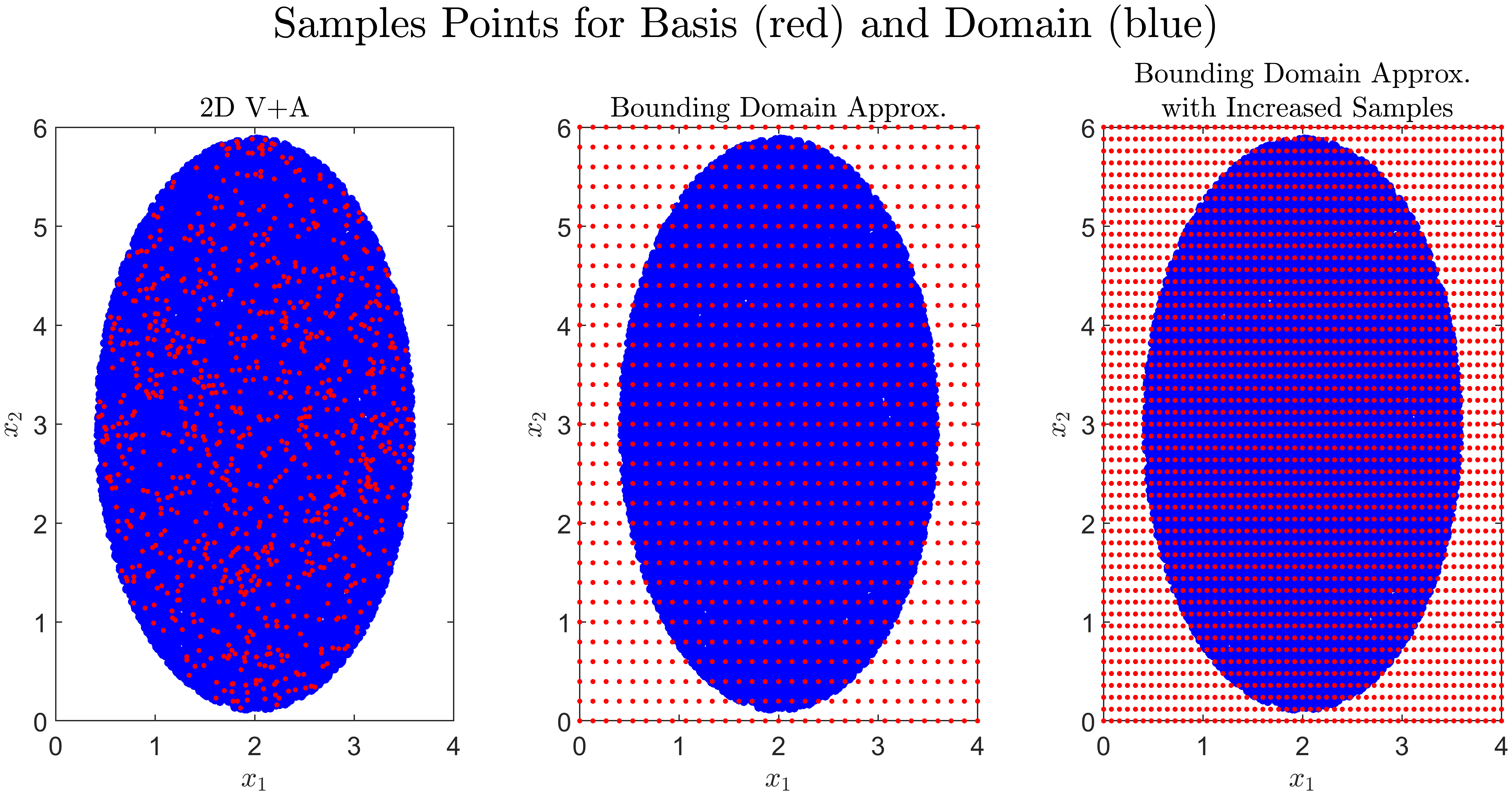}
    \caption{\small {An illustration of the sample points where discrete orthogonal polynomials are generated for Example \ref{example 3.4}. In V+A, the bases are generated by sample points from the elliptical domain, represented by red dots in the first plot. In the orthogonal basis on the bounding domain approximation, the bases are Legendre polynomials on the tensor product domain generated by red dots in the second and third plots. The first and second plots have the same number of sample points.  The first and third plots have the same number of sample points inside the elliptical domain. }}
    \label{figure illustration of sample points}
\end{figure}

\begin{example} 
    \label{highD example}
{\textbf{(V+A Polynomial Approximation for Higher Dimensional Multivariate Functions)} Figure \ref{fig highD example} presents preliminary numerical experiments on using V+A for high-dimensional function approximation (specifically, for $d>2$). In the cases of $d=3$ and $d=5$, the least squares approximation using the V+A algorithm consistently reduces error for both tensor product and irregular domains, while the Vandermonde approximation is inaccurate due to ill-conditioning.
These examples highlight the effectiveness of least-squares approximations using the V+A algorithm in higher-dimensional multivariate functions and non-tensor product domains. It is important to note that the multivariate V+A algorithm (Algorithm \ref{Multi V+A Algo}) and the theories we prove in this paper for multivariate V+A  are applicable for $d\ge 2$. However, due to the growth of the polynomial basis with the dimensionality of functional approximation (i.e., $N = \O(n^d)$ for both polynomial spaces), the CPU time required for numerical experiments on higher-dimensional multivariate functions is affected by the curse of dimensionality. 
To overcome this, one may need to combine V+A with e.g. sparse grids~\cite{bungartz2004sparse}. The design of practical and scalable algorithms for high-dimensional functional approximation is a task we defer to future work.}

\begin{figure}[!ht]
    \centering
\includegraphics[width=5cm, height = 3.9cm]{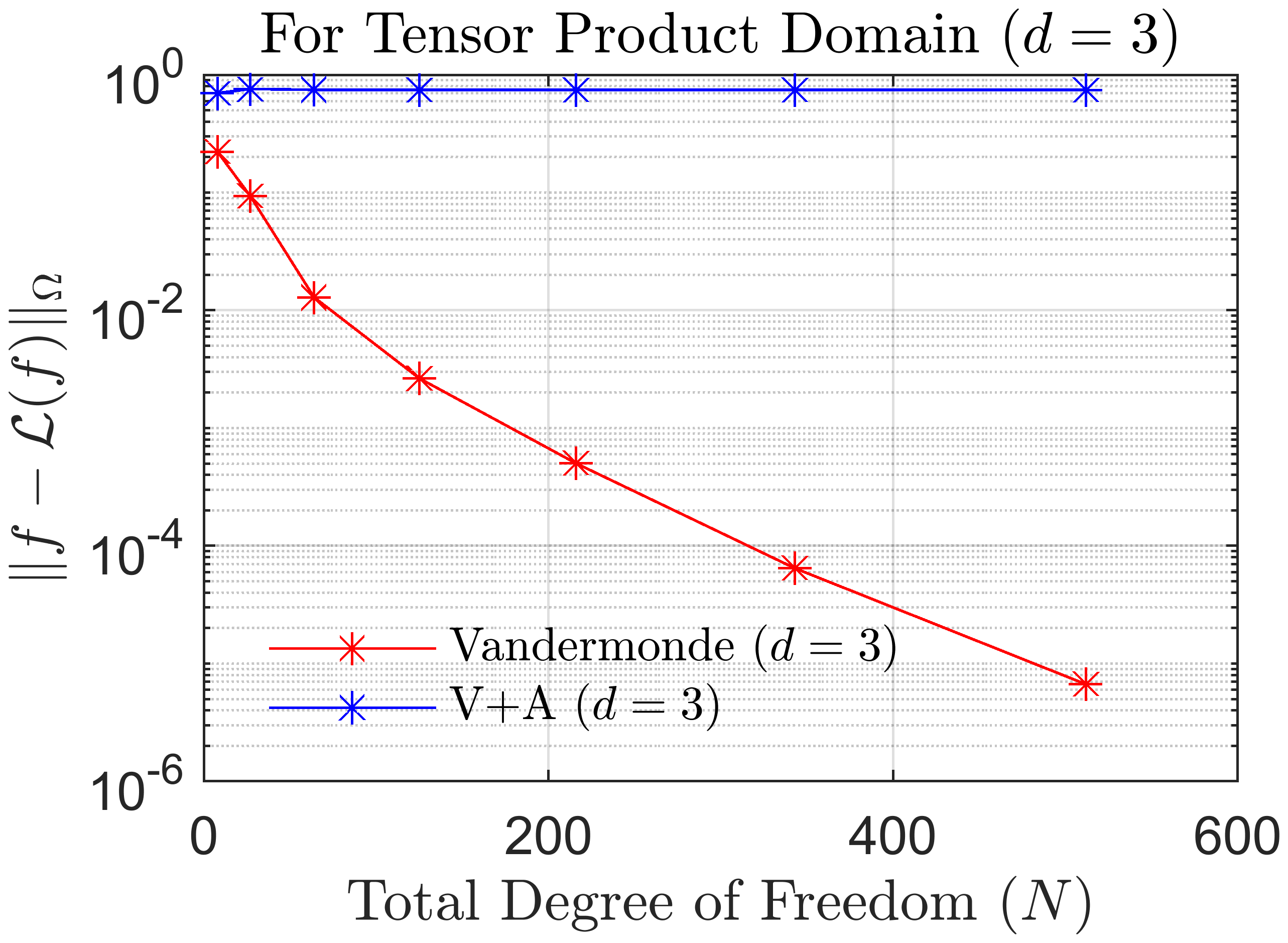}
\includegraphics[width=5cm, height = 3.9cm]{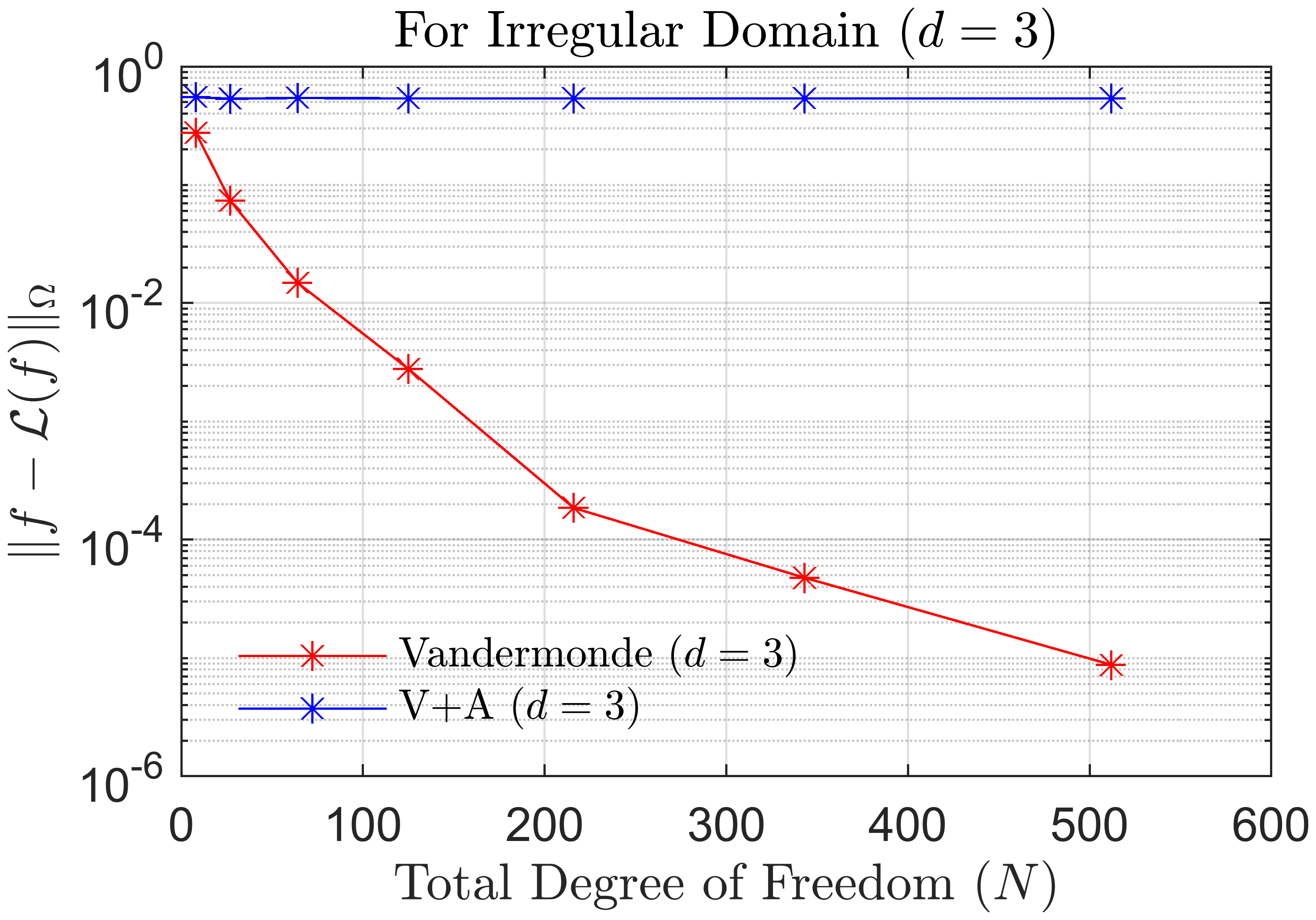}
\includegraphics[width=5cm, height = 3.9cm]{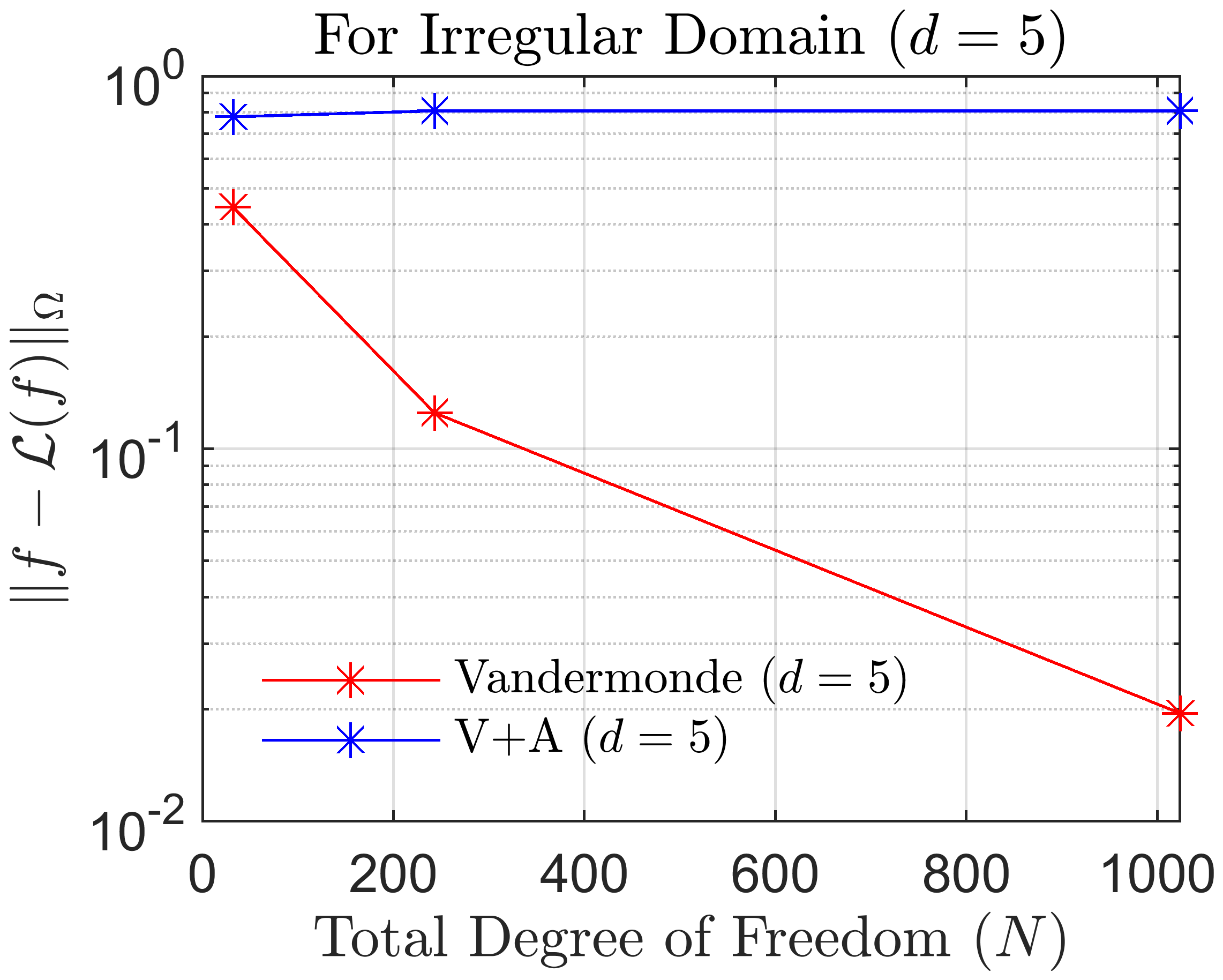}
\caption{\small Approximating $f(\mathbf{x})=\sin(\|\mathbf{x}\|^2)$ with $\mathbf{x} \in \R^d$, where $d=3, 5$, using $p \in \P^M_{d,n}$. The tensor product domain is defined as $[0, 1]^d$, and the irregular domain is defined as $\|\tilde{\mathbf{w}}^T\mathbf{x}\| \le 1$, where $\tilde{\mathbf{w}} \in \R^d$ with entries uniformly distributed in $[0, 1]$. {The sample points are equispaced points in the domain; additional details on sample complexity are discussed in Section \ref{sec: sampling}.}}
\label{fig highD example}
\end{figure}
\end{example}

\section{Sample Complexity and Convergence of Multivariate V+A}
\label{Section leb constant}

In this section, we establish the sample complexity and convergence properties of the least squares approximation using {the V+A algorithm for the total degree space. Note that the analysis follows naturally for the maximum degree of polynomial space, up to $d$-dependent constants.} We demonstrate that, across many domains, utilizing either $M = \O(N^2)$ equispaced points or $M = \O(N^2 \log N)$ random sample points, the V+A approximant with $N$ degrees of freedom is near-optimal. Namely, the error of the approximation is within a polynomial factor of that of the best approximation, for instance, $\|f-\mathcal{L}(f)\|_{\Omega}= \O(N)\|f-p^*\|_{\Omega}$ where {the $\O(\cdot)$ notation involves constants that depend on $d$.}
Furthermore, we prove that, when the sample complexity condition is satisfied, the error of the least squares approximation using the V+A algorithm converges at a spectral rate. Specifically,
\begin{equation}
     \|f - \mathcal{L}(f)\|_{\Omega}  = \O(n^{d-k})
\end{equation}
where $d$ is the dimension of the variables (i.e., $x \in \R^d$), and $k$ is the highest order of the derivative of $f$ that exists and is continuous (i.e., smoothness).

\noindent

\textbf{The proofs and key contributions for this section are structured as follows.}

\begin{enumerate}
\item  We first prove that convex bodies or finite unions of convex bodies satisfy the Markov property {for total degree polynomial space (and also maximum degree space)}, such that
\begin{equation}
\tag{Markov Condition}
 \|\nabla p(\x)\|_{\Omega}:= \sup_{\x \in \Omega} \bigg(\max_{1 \le r \le d} \bigg|\frac{\partial p}{\partial x_{(r)}}(\x) \bigg|\bigg) \le \M(\Omega) n^2 \| p \|_{\Omega}
\label{Markov property with degree 2}
\end{equation}
where $\M(\Omega)$ is a domain-dependent constant. We highlight that this Markov property for the domain and the polynomial space is crucial for proving the sample complexity. 
 {Note that the constant for the Markov property (for total degree polynomial space and maximum degree space) differs by a factor depending on $d$ (Section \ref{sec: Markov}). }
 
\item Under the Markov property condition, we can form a set of $M = \O(n^{2d}) = \O(N^2)$ equispaced sample points, such that the maximum absolute value of the polynomial in $\Omega$ is bounded by a constant multiplying the maximum absolute value of the polynomial at sample points $\boldsymbol{X}$, $\|p\|_{\Omega} \le 2\|p\|_{\X}$ (Section \ref{sec: sampling}). 
To the best of our knowledge, the Markov property and complexity results of this kind have only been established for total polynomial spaces in the literature \cite{Markov, kellogg1928bounded, kroo2010approximation, wilhelmsen1974markov}, but not specifically for maximum degree space.

\item Then, we demonstrate V+A gives a near-optimal convergence rate (Section \ref{sec Convergence of Multivariate Vandermonde with Arnoldi}) {for the total degree polynomial spaces}\footnote{
{In exact arithmetic, the V+A approximation is the same as that of a least-squares solution based on any basis, as long as the span of the bases is the same. The convergence results that therefore apply to any method based on least-squares, without the V+A orthogonalization. Of course, in finite precision, V+A can make a dramatic difference. For this reason, while the convergence results would hold for a least-squares method employing any basis, we often specify the use of V+A, as otherwise the theory may not be reflected in actual computations. 
}}.

\item At the end of the section, we provide a comparison between the sample complexity of the least squares approximation using the V+A Algorithm and the sample complexity of polynomial frame approximation~\cite{frame2, Adcock2019, cohen, cohenweight, migdisorth} (Section \ref{sec Related Work on Sample Complexity}). 

\end{enumerate}

\subsection{Markov Inequality}
\label{sec: Markov}

In this section, we demonstrate that \eqref{Markov property with degree 2} holds for a variety of domains. We start with the simplest domain with $d=1$. For univariate polynomial space, the two polynomial spaces overlap, $\P^M_{1,n} = \P^T_{1,n}.$ Thus, the Markov condition for $d=1$ is given in  \cite{wilhelmsen1974markov}. 
\begin{lemma}
\label{Lemma 3.1}
\textbf{{(Markov Condition for $d=1$)}}
For $x \in [-\zeta, \zeta]$ with $\zeta>0$, then 
$
\big|p'(\x)\big| \le \frac{n^2}{\zeta} \max_{x \in [-\zeta, \zeta]}|p(\x)|. $
\label{1d markov}
\end{lemma}

For $d > 1$, we show that the Markov condition also holds using Lemma \ref{lemma markov} and Theorem \ref{Theorem markov}.

\begin{lemma} \textbf{{(Markov Condition for $d>1$)}}
\label{lemma markov}
Let $\Omega \subset \R^d$ be a convex body, i.e., a convex compact set with a non-empty interior. 
{Let $\boldsymbol{\alpha} = [\alpha_1, \dotsc, \alpha_d]\in\mathbb{N}_0^{d}$ and $p$ be a polynomial consisting of terms of at most $x_{(1)}^{\beta_1} \dotsc x_{(d)}^{\beta_d}$ with $\beta_i \le  \alpha_i$ for $ 1 \le i \le d$.}
Then, the following results hold.
\begin{enumerate}
    \item Let $\x^* \in \Omega$ such that $\|\nabla p(\x^*)\|_2 =     \|\nabla p \|_{\Omega, 2} = \sup_{\x \in \Omega} \|\nabla p(\x)\|_2$. Assume also $p(\x^*)=1$, then 
$$
            \|\nabla p \|_{\Omega} \le \frac{2\|\boldsymbol{\alpha}\|_1^2}{\omega(\Omega)}.
            \label{3 space bound}
$$
Here $\omega(\Omega)$ is the minimal distance between two parallel supporting hyperplanes for $\Omega$ and {$\|\boldsymbol{\alpha}\|_1 = \sum_{i=1}^d \alpha_i$.} 

\item For any multivariate polynomial $p(\x)$ with order $n$ and $\|p\|_{\Omega}\le 1$ without the condition $p(\x^*)=1$, a similar inequality holds true
$
\|\nabla p \|_{\Omega} \le \frac{4\|\boldsymbol{\alpha}\|_1^2}{\omega(\Omega)}. 
$
\end{enumerate}
\end{lemma}

\begin{proof}
We start by proving the first result. 
We omit the trivial case when $p$ is constant. Under the assumption that  $\|\nabla p(\x^*)\|_2 =     \|\nabla p \|_{\Omega, 2} = \sup_{\x \in \Omega} \|\nabla p(\x)\|_2$, if  $p$ is not a constant, this implies that  $\|\nabla p(\x^*)\|_2 >0 $, we must have $\x^*$ on the boundary  $\x^* \in \partial \Omega$. 

Let $\u = \frac{\nabla p(\x^*) }{\|\nabla p(\x^*)\|_2}  $, and let $P_{\u}$ be
the hyperplane with normal {$\u$} which passes through $\x^*$ .
\cite[Lemma 3.1]{wilhelmsen1974markov} proved that  $P_{\u}$ is a supporting hyperplane of $\Omega$. Namely, 
$
P_{\u} = \{\x \in \R^d, (\x-\x^*)^T \u=0\}
$
and $P_{\u}$ contains no interior points of $\Omega$.
Now, assume $\u$ is an outer normal for $\Omega$, and let $P_{-\u}$ be the supporting hyperplane with outer normal $-\u$ lying a distance $\lambda_\u$
from $P_{\u}$. Let $\x^0 \in P_{-\u} \cap \Omega$. The line segment $[\x^0, \x^*]$ lies in $\Omega$ by convexity, and
\begin{align}
    \text{$p(\x^* + \lambda \w)$, where $\w = \frac{(\x^0 -\x^*)}{\|\x^0 -\x^*\|_2}$, is a polynomial in $\lambda$ with order $\|\boldsymbol{\alpha}\|_1$.}
    \label{key diff step}
\end{align}
This polynomial is bounded by $1$ for $\lambda \in \big[0, \|\x^0 -\x^*\|_2 \big]$. Since polynomials are translation invariant, Lemma \ref{1d markov} yields
\begin{eqnarray*}
    \frac{2\|\boldsymbol{\alpha}\|_1^2}{ \|\x^0 -\x^*\|_2} \ge \bigg|\frac{d}{d \lambda} p(\x^* + \lambda \w)\bigg| =\bigg|\nabla p(\x^*)^T \w\bigg| = \frac{\lambda_\u}{ \|\x^0 -\x^*\|_2} \| \nabla  p(\x^*) \|_2
\end{eqnarray*}
where the last equality comes from $\nabla  p(\x^*) ^T (\x^0 -\x^*) = \lambda_\u \nabla  p(\x^*) ^T \u = \lambda_\u \|\nabla  p(\x^*) \|_2$. Also, since $\omega(\Omega)$ is the minimal distance between two parallel supporting hyperplanes for $\Omega$, we have {$ \lambda_\u \ge \omega(\Omega) >0$}. Therefore,
$
    \frac{2\|\boldsymbol{\alpha}\|_1^2}{ \|\x^0 -\x^*\|_2} \ge \frac{\omega(\Omega) }{ \|\x^0 -\x^*\|_2}  \| \nabla  p(\x^*) \|_2.
$
By the definitions of $\Omega$-norm and $\x^*$, we deduce that
$\| \nabla  p(\x^*) \|_{\Omega} \le \| \nabla  p(\x^*) \|_{\Omega, 2} \le  \frac{2\|\boldsymbol{\alpha}\|_1^2}{\omega(\Omega) }$. Extending the first result to the second result follows the same proof as in \cite[Thm 3.1]{wilhelmsen1974markov}. 
\end{proof}

\begin{remark}
Lemma \ref{lemma markov} is adapted from \cite[Lemma 3.1 \& Thm 3.1]{wilhelmsen1974markov}. In \cite{wilhelmsen1974markov}, the theorem was originally designed for the total degree polynomial space, and the bounds were given as:
$$
    \|\nabla p \|_{\Omega} \le \frac{2{n^2}}{\omega(\Omega)}
, \qquad
    \|\nabla p \|_{\Omega} \le \frac{4{n^2}}{\omega(\Omega)}
$$ for Results 1--2, respectively. 
When the polynomial space is the total degree space, $\|\alpha\|_1 = n$, and Lemma \ref{lemma markov} reduces to \cite[Lemma 3.1 \& Thm 3.1]{wilhelmsen1974markov}.
Our theorem also applies to the maximum degree where  $\|\alpha\|_1 = nd$.
\end{remark}

\begin{theorem}
Let $d\ge 1$, and assume $\Omega \subset \R^d$ is a convex body. The multivariate polynomial $p$ with order $n$ satisfies the Markov condition in $\Omega$, 
\begin{equation}
\tag{Markov Condition}
\|\nabla p \|_{\Omega} \le \M(\Omega) n^2 \|p\|_{\Omega}
    \label{markov}
\end{equation}
where $\M(\Omega) = \frac{4 d^2}{\omega(\Omega)}$ for the maximum degree space, and $\M(\Omega) = \frac{4}{\omega(\Omega)}$ for the total degree space. $\omega(\Omega)$ is the minimal distance between two parallel supporting hyperplanes for $\Omega$. 
\label{Theorem markov}
\end{theorem}
\begin{proof}
This is a direct application of Lemma \ref{lemma markov} up to a scaling function of $ \| p \|_{\Omega}$. $M(\Omega)$ is obtained by the expression of $\|\boldsymbol{\alpha}_1\|$ in \eqref{a1 1}--\eqref{a1 2}.
\end{proof}

\subsection{Sampling}
\label{sec: sampling}

We prove in this section that if $M = \O(n^{2d}) = \O(N^2)$ equispaced sample points are chosen from a domain of a convex body or finite unions of convex bodies, then the maximum absolute value of the polynomial in $\Omega$ is bounded by a constant multiplying the maximum absolute value of the polynomial at sample points ${\boldsymbol{X}}$, $  \|p\|_{\Omega} \le 2\|p\|_{\X}.  $
Since the constant for the Markov property $\M(\Omega)$ only differs by a factor depending on $d$ (Theorem \ref{Theorem markov}), the same order of sample complexity results apply to {both the maximum degree and total degree spaces up to a $d$-dependent constant}. 
In a high-dimensional domain, it is sometimes easier to generate sample points randomly from the domain by rejection sampling. Therefore, in Section \ref{Randomized sample points}, we prove that with $\O(N^2 \log N)$ random sample points $\tilde{\X}$ generated i.i.d. from a uniform measure, $  \|p\|_{\Omega} \le 2\|p\|_{\tilde{\X}}$ also holds with high probability {for the maximum degree and total degree spaces up to a $d$-dependent constant}.

 \subsubsection{Deterministic Sample Points} 
  \label{Deterministic sample points} 

{We define equispaced discretization in Definition \ref{def 1}} and the sample complexity for equispaced sample points is given in Theorem \ref{thm 3.3}. 

{

\begin{definition}
\textbf{(Equispaced Discretization)}  
Let $\Omega \subset \mathbb{R}^n$ be a compact domain and $\Omega_{cube} \subset \mathbb{R}^n$ a bounding box containing $\Omega$. An \textit{equispaced discretization} of $\Omega$ is constructed by first defining an equispaced grid of points on $\Omega_{cube}$, and subsequently restricting this grid to the points contained within $\Omega$. We write $\{\x_i\}_{i=1}^N$ for the set of equispaced points in $\Omega_{cube}$. 
\label{def 1}
\end{definition}}

\begin{theorem}
\label{thm 3.3}
Let $\Omega \subset \R^d$ with $d>1$ be a convex body {(i.e., a convex compact set with a non-empty interior)} or finite unions of convex bodies. Let $0 < c \le \frac{1}{2\sqrt{d}}$ be a fixed constant.
Construct a set of $M$ sample points ${\boldsymbol{X}} = \big\{ \x_i |\x_i \in \Omega \subset \R^d \big\}_{1\le i \le M}$ such that for all $\x \in \Omega$, we have
\begin{equation}
\tag{Mesh Condition}
\inf_{\x_i \in \X} \|\x_i-\x\|_{\infty} \le \frac{c}{\M(\Omega) n^2}   \label{mesh}
\end{equation}
where {$\M(\Omega) = \frac{4 d^2}{\omega(\Omega)}$} for the maximum degree space, and {$\M(\Omega) = \frac{4}{\omega(\Omega)}$} for the total degree space. Then, 
\begin{eqnarray}
   \|p\|_{\Omega} \le 2\|p\|_{\X}
   \label{bound for polynomial at sample points}
\end{eqnarray}
for all multivariate polynomials. {Using equispaced discretization as in Definition \ref{def 1}}, the cardinality of $|\X|$ satisfies $|\X| = \O(n^{2d}) = \O(N^2)$, where $N$ is the number of basis functions. {Note that the $\O(\cdot)$ notation involves constants that depend on $d$.}
\end{theorem}

\begin{proof}
The proof is similar to \cite{Markov} (3.13)--(3.18). We have sharpened the upper bound for $c$ in \eqref{mesh} by a factor of $\sqrt{d}$.  For any fixed $p \in \P_{d,n}$  (i.e., maximum degree, total degree), let $ \x_*  = [{(x_*)}_1, \dotsc, {(x_*)}_d]^T\in \Omega$ such that the maximum value is attached at this point $\|p\|_{\Omega} = |p(\x_*)|$ and for all $1 \le i \le M$ we have $\x_i  \in \X$ such that $\|\x_*-\x_i\|_{\infty} \le \frac{c}{\M(\Omega)n^2}$. Then, using the mean value theorem, there exists $\boldsymbol{c}$ on the line connecting $\x_i$ to $\x_*$ such that
\begin{eqnarray}
 \big|p(\x_*)-p(\x_i)\big| = |\nabla p(\boldsymbol{c})^T (\x_*-\x_i)|
 \le \big\|\nabla p\big\|_{[\x_i, \x_*]} \big\|\x_*-\x_i\big\|_2  \le \sqrt{d} \big\|\nabla p\big\|_{[\x_i, \x_*]} \big\|\x_*-\x_i\big\|_{\infty}
 \label{sharpen bound}
\end{eqnarray}
where $[\x_i, \x_*]$ represents the line connecting $\x_i$ to $\x_*$ and $\big\|\nabla p\big\|_{[\x_i, \x_*]} =  \sup_{\x \in [\x_i, \x_*]} \bigg(\sum_{r=1}^d \big|\frac{\partial p}{\partial x_{(r)}}(\x) \big|^2\bigg)^{1/2}$.  

Since $\Omega$ is convex, we have $[\x_i, \x_*] \subset \Omega$. Theorem \ref{Theorem markov} gives that $\big\|\nabla p\big\|_{[\x_i, \x_*]} \le \M(\Omega)n^2 \|p\|_{\Omega}$. Together with  \eqref{mesh}, we deduce that
 \begin{eqnarray*}
 \big|p(\x_*)-p(\x_i)\big|\le \sqrt{d} \bigg(\M(\Omega)n^2 \|p\|_{\Omega} \bigg) \bigg(\frac{c}{\M(\Omega)n^2}\bigg) \le \sqrt{d} c \|p\|_{\Omega} \le \frac{1}{2} \|p\|_{\Omega}. 
\end{eqnarray*}
where the last inequality comes from $0 < c \le \frac{1}{2\sqrt{d}}$. Hence, 
$$
\big\|p\big\|_{\Omega} - \big|p(\x_i)\big| =     \big|p(\x_*)\big|-\big|p(\x_i)\big|  \le     \big|p(\x)-p(\x_i)\big|  \le \frac{1}{2} \|p\|_{\Omega}.  
$$
We deduce that $\|p\|_{\Omega} \le 2 \|p\|_{\X}$.
To prove the second argument regarding the cardinality of the mesh, we construct the mesh $\X$ using a standard
discretization process. For $\Omega \in \R^d$, we enclose $\Omega$ in a $d$-dimensional hypercube $\Omega_{cube}\in \R^{d}$. We then cover $\Omega_{cube}$ with equispaced grids dense enough such that \eqref{mesh} is satisfied. This gives us $\O(n^{2d})=\O(N^{2})$ number of nodes for the mesh. 
Lastly, the extension from a convex domain to a domain with the finite union of convex bodies is a direct application of \cite[Lemma 4(i)]{Markov}. 
\end{proof}

\begin{remark} \textbf{(Equispaced Sampling Method)}
Definition \ref{def 1} and Theorem \ref{thm 3.3} provide us with a way to discretize the domain and choose sample points that can control the size of the discrete orthogonal polynomial. For $\Omega \in \R^d$, we enclose $\Omega$ in a $d$-dimensional hypercube $\Omega_{cube}\in \R^{d}$. We then cover $\Omega_{cube}$ with equispaced grids dense enough such that \eqref{mesh} is satisfied which gives us $\O(n^{dr})=\O(N^{r})$ number of nodes.
\label{sample method}
\end{remark}

Note that the deterministic sampling method can be further improved in two ways. Firstly, instead of sampling in the hypercube domain  $\Omega_{cube}$, one could sample directly on $\Omega$ using rejection sampling. Details on the randomized sampling method are discussed in Section \ref{Randomized sample points}. Secondly, instead of taking equispaced samples, we can choose a set of sample points with a `near-optimal weight'. Details on near-optimal sampling strategy are discussed in Section \ref{section Weighted LSA}.

\subsubsection{Randomized Sample Points} 
\label{Randomized sample points} 
Theorem \ref{Ramdon AM} gives the sample complexity for random sample points. 

\begin{theorem}\textbf{(Randomized Sampling Method)}
Let $\Omega \subset \mathbb{R}^d$ with $d > 1$ be a convex body or finite union of convex bodies. Let $\tilde{\X}=\{\tilde{\x}_1, \tilde{\x}_2, \dots, \tilde{\x}_{\tilde{M}}\}$ be a set of points independently and uniformly sampled from $\Omega$. Let $\tilde{c} > 1$, and set ${\tilde{M}} \ge \tilde{c} M\log(M)$. Then, $ \|p\|_{\Omega} \le 2\|p\|_{\tilde{\X}}$ with probability at least $1-M^{-\tilde{c}}$ for the total degree polynomial spaces\footnote{The same order of result is true for the maximum degree polynomial spaces, up to a $d$-dependent constant.}. 
\label{Ramdon AM}
\end{theorem}

\begin{proof}
We prove this theorem using a similar approach to \cite[Thm 4.1]{RandomizedXu}. For the set of random sample points $\boldsymbol{\tilde{X}}$ to satisfy \eqref{mesh}, we need to prove that with high probability for every $\x \in \Omega$, one can find a point in $\boldsymbol{\tilde{X}}= \{\tilde{\x }_i\}_{1\le i \le \tilde{M}}$ such that their distance is at most  $E: = \frac{c}{\M(\Omega) n^2} $. Specifically, for every $\x \in \Omega$, $\boldsymbol{\tilde{X}} \cap \bar{B}_{E}(\x )  \neq \varnothing$ is satisfied with high probability where  $\bar{B}_{E}(\x_i)$ denotes the closed ball with center $\x_i$ and radius $E$. We have 
\begin{align*}
\mathbb{P} \left[ \boldsymbol{\tilde{X}} \cap \bar{B}_{E}(\x ) \neq \varnothing \right] = 1-\big( \mathbb{P}\left[ \tilde{\x}_i \cap \bar{B}_{E}(\x ) = \varnothing \right]\big) ^{\tilde{M}} \ge 1- \left(1 - |\X| ^{-1}\right)^{\tilde{M}} \ge 1- \left(1 - M^{-1}\right)^{\tilde{M}} \ge 1-e^{-\tilde{M}/M}. 
\end{align*}
The first equality follows from the i.i.d sampling nature of the random sample points. Note that $|\X| = M $ denotes the cardinality of $\X$. The first inequality is given in the proof of \cite[Thm 4.1]{RandomizedXu}. By taking $\tilde{M} \ge \tilde{c} M\log(M)$ with $\tilde{c}>1$, we achieve $\boldsymbol{\tilde{X}} \cap \bar{B}_{E}(\x ) \neq \varnothing $ with probability at least $1-M^{-\tilde{c}}$. Applying Theorem \ref{thm 3.3} completes the proof.  
\end{proof}



\subsection{Convergence of Multivariate Vandermonde with Arnoldi}
\label{sec Convergence of Multivariate Vandermonde with Arnoldi}

Theorem \ref{thm 3.3} and Theorem \ref{Ramdon AM} can be readily used to prove the convergence results for V+A.

\subsubsection{Error Bound}
\begin{theorem} 
Let $f:  \Omega \rightarrow {\R^d}$ be a continuous multivariate function. Let $\Omega \subset \R^d$ with $d>1$ be a convex body or finite unions of convex bodies. Let {$p^* \in \P^T_{d,n}$ be the best polynomial approximation of $f$ in the total degree polynomial space, such that $\|f-p^*\|_{\Omega}\leq\|f-p\|_{\Omega}$ for all {$p \in  \P^T_{d,n}$}. Construct a set of equispaced sample points as described in Remark \ref{sample method}, ${\boldsymbol{X}} = \big\{ \x_i \big\}_{1\le i \le M} \subset \Omega$ with  $M = \O(N^2)$.
{Applying} Algorithm \ref{Multi V+A Algo} using the sample points $\X$, the error of the least squares approximation satisfies\footnote{The bound for the maximum degree space can be deduced similarly up to a $d$-dependent constant.}}
\begin{equation}
\|f-\mathcal{L}(f)\|_{\Omega} \le \left(1+  \frac{2}{M}  \|\Q  \Q^*\|_{\infty}\right) \|f-p^*\|_{\Omega} \le \bigg( 1+\O(N)\bigg) \|f-p^*\|_{\Omega}
\label{kate bound}
\end{equation}
{where the columns of $\Q$ are the values of the discrete orthogonal basis generated by V+A at the sample points $\X$.} The matrix $\infty$-norm is defined as the largest absolute value of the matrix.  
\label{thm 4.3}
\end{theorem}

\begin{proof} 
Denote {$g:=f-p^*$}. We have 
\begin{equation}
\|f - \mathcal{L}(f)\|_{\Omega} \le \|\underbrace{f- p^*}_{:= {g}} \|_{\Omega}+ \|\underbrace{\mathcal{L}(f-p^*)}_{= {\mathcal{L}(g)}\in {\P^T_{d,n}}}\|_{\Omega}.
\label{thm 4.3 1}
\end{equation} 
Since $\mathcal{L}(g)\in {\P^T_{d,n}}$, we write it as a linear combination of discrete orthogonal polynomials $\mathcal{L}(g):= \sum_{j=1}^N \beta_j \phi_j=\Q  \boldsymbol{\beta}$ where $\boldsymbol{\beta}:=[\beta_1, \dotsc, \beta_N]^T\in \R^N. $ According to \eqref{sol d}, $\boldsymbol{\beta}=\frac{1}{M}\Q^* {\tilde{g}}$ {where  $\tilde{g} \in \R^M$ is a vector with entries as $\{f(\x_i)-p(\x_i)\}_{1 \le i \le M}$.}  Using Theorem \ref{thm 3.3}, we have
\begin{eqnarray}
\|\mathcal{L}(g)\|_{\Omega}  \le 2  \|\mathcal{L}(g)\|_{\boldsymbol{X}} = 2 \|\Q  \boldsymbol{\beta}\|_{\infty}= \frac{2}{M} \|\Q  \Q^* \tilde{g}\|_{\infty}. 
\label{thm 4.3 2}
\end{eqnarray}
Substituting \eqref{thm 4.3 2} into \eqref{thm 4.3 1} gives the first inequality in \eqref{kate bound}. Using the matrix property $    \frac{1}{M} \|\Q  \Q^*\|_{\infty} < \frac{1}{\sqrt{M}}\|\Q  \Q^*\|_{2}= \sqrt{M}$ and the sample complexity assumption $M = \O(N^2)$, we arrive at the second inequality in \eqref{kate bound}.
\end{proof}

\begin{remark}
A similar bound was derived by {Calvi} and Levenberg in~\cite[Thm 2]{Markov} using Hermitian inner product norm, 
\begin{equation}
     \|f - \mathcal{L}(f)\|_{\Omega} \le \bigg(1+2(1+\sqrt{M})\bigg) \|f-p^*\|_{\Omega}.
    \label{mesh bound}
\end{equation}
Since $\frac{1}{M}\|\Q  \Q^*\|_{\infty} < \frac{1}{\sqrt{M}}\|\Q  \Q^*\|_{2}= \sqrt{M}$, it is straightforward to verify that our bound in Theorem \ref{thm 4.3} gives a sharper approximation than the bound in \eqref{mesh bound} by {Calvi} and Levenberg.
\end{remark}

\begin{example} 
\label{better numerical results}
\textbf{(Improved Numerical Error Bound on Certain Domains)} Due to the special properties of the V+A basis (i.e., $\Q $), for the intervals, 2D tensor product domain, and other domains (as indicated in Figure \ref{QQT}), our numerical examples suggest that the {least-squares approximant} gives a tighter error bound than \eqref{kate bound} in Theorem \ref{thm 4.3}, 
$$
\|f-\mathcal{L}(f)\|_{\Omega} \approx (1+\sqrt{2N}) \|f-p^*\|_{\Omega}. 
$$
In other words, with $M=\O(N^2)$ equispaced sample points, $\frac{1}{M}\|\Q  \Q^*\|_{\infty}$ scales like $\sqrt{2N}$ in the domains that we tested. 
{
For intervals and 2D tensor product domains the discrete orthogonal basis generated by $M = \mathcal{O}(N^2)$ equispaced sample points resembles the scaled Legendre polynomials, such that $\tilde{L}_j(\x) \approx \phi_j(\x)$ for $d=1$ and $\tilde{P}_j(\x) \approx \phi_j(\x)$ for $d \ge 2. $ The scaled Legendre polynomials for $d=1$ and $d \geq 2$  are defined as follows.

\begin{definition} {(Univariate Scaled Legendre Polynomials)} For any interval $[a ,b] \in \R$ with $a<b$,  we define the scaled univariate Legendre polynomials as
\begin{equation*}
   \tilde{L}_j(\x) = \gamma_j L_j(\eta(\x)), \qquad  \text{with}\qquad \gamma_j = \sqrt{2j-1}.
\end{equation*}
$\eta(\x) := \frac{2(\x-a)}{b-a}-1$ is the linear map from $[a, b]$ to $[-1, 1]$ and $L_j$ is the $j$th standard univariate Legendre polynomials. The scaled Legendre polynomials are orthonormal in $[a ,b]$, such that $  \int_{[a ,b]}\tilde{L}_j({x})\tilde{L}_k({x})d{x} = \delta_{j,k}$ for any $j,k=1,2,\dotsc$. 
\end{definition}

\begin{definition} {(Multivariate Scaled Legendre Polynomials)} For tensor--product domain $\Omega_d:=[a_{(1)}, b_{(1)}]\times \dotsc \times[a_{(d)}, b_{(d)}] \in \R^d$ with $a_{(r)} < b_{(r)}$ for $r = 1, \dotsc, d$,  we define the scaled multivariate Legendre polynomials as
\begin{equation*}
   \tilde{P}_j(\x) =  \Gamma_j P_j(\eta(\x)), \qquad  \text{with}\qquad  \Gamma_j=\prod_{r=1}^d \sqrt{2j_{(r)}+1}.
\end{equation*}
$\eta(\x) = (\eta_1(x_{(1)}), \dotsc, \eta_d(x_{(d)}))$ with $\eta_r(x_{(r)}) := \frac{2(x_{(r)}-a_{(r)})}{b_{(r)}-a_{(r)}}-1$ maps $\Omega_d$ to $[-1,1]^d$. Note that $P_j$ is the standard multivariate Legendre polynomials
\begin{equation*}
   P_j(\x)= \prod_{r=1}^d L_{j_{(r)}+1}(x_{(r)}), \quad  \x = (x_{(1)}, \dotsc, x_{(d)})\in [-1,1]^d. 
\end{equation*}
$j_{(r)}$ is the degree for the $r$th entry $r = 1, \dotsc, d$. The scaled Legendre polynomials are orthonormal in $\Omega_d$, such that $  \int_{\Omega_d}\tilde{L}_j({x})\tilde{L}_k({x})d{x} = \delta_{j,k}$ for any $j,k=1,2,\dotsc$. 
\end{definition}}

We give a brief justification for the tighter error bound than \eqref{kate bound} in these cases. Firstly, we rewrite $\frac{1}{M}\|\Q  \Q^*\|_{\infty}$ in expanded form
\begin{align*}
\frac{1}{M} \max_{1\le k \le M} \sum_{i=1}^M \left|\sum_{j=1}^N {\phi_j(\x_k)}\phi_j(\x_i)\right| \le \underbrace{\left(\max_{1 \le j \le N,1 \le k \le M}|{\phi_j(\x_k)}|  \right)}_{:=Q_{\max}} \underbrace{\left(\frac{1}{M} \sum_{i=1}^M \left|\sum_{j=1}^N \phi_j(\x_i)\right|\right)}_{:=S_N}. 
\end{align*}
The discrete orthogonal polynomials generated $M=\O(N^2)$ equispaced points in real intervals are well approximated by the scaled Legendre polynomials. Therefore, $Q_{\max} $ is well approximated by the size maximum absolute value of the scaled Legendre polynomials $Q_{\max} \approx \max_{1 \le j \le N}\sqrt{2j-1} = \sqrt{2N-1}$. On the other hand, $\sum_{j=1}^N \phi_j(\x_i)$ is small in the majority of the interval and only takes large values near the endpoints (Figure \ref{QQT2}). The supremum of $\sum_{j=1}^N \phi_j(\x_i)$ has a scaling of  $\sum_{j=1}^N\sqrt{2j-1}\approx\frac{(2N-1)^{3/2}}{3}$ in an interval of length $M^{-1}$. The mean value of the absolute sum of the discrete orthogonal polynomials, $S_N$, is relatively constant as $N$ grows. Our numerical experiment shows that $S_N \in [1.2, 1.3]$ when $N$ increases from $10$ to $300$. As a result, $Q_{\max}S_N$ together gives an upper bound for $\frac{1}{M}\|\Q  \Q^*\|_{\infty}$, that is, $\O(\sqrt{2N-1})$. Similar arguments follow for the 2D domains and the bound in 2D is obtained using $\max_{j,k}|{\phi_j(\x_k)}| \approx  \sqrt{2N}.$
\end{example}

{

\sidecaptionvpos{figure}{c}
\begin{SCfigure}[50][ht]
    \centering
    \includegraphics[width=9cm]{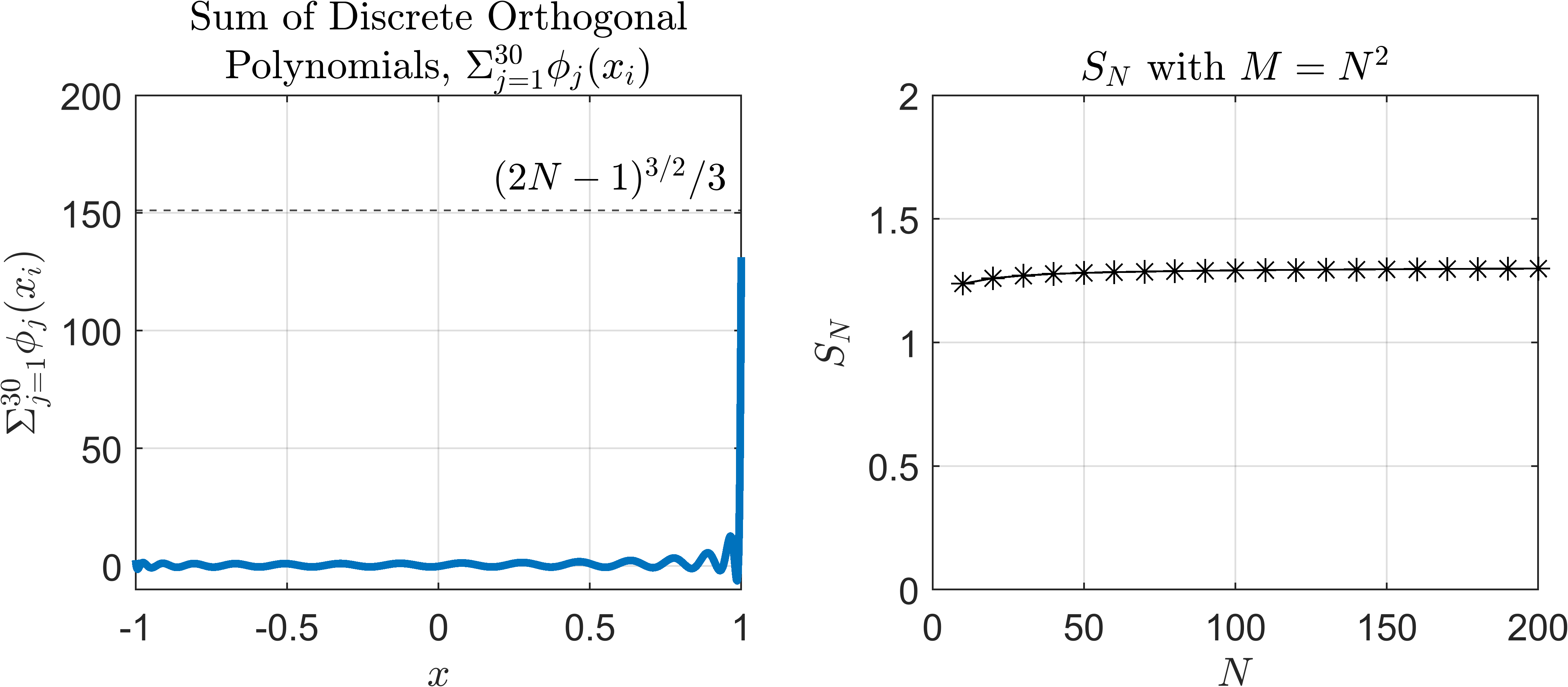}
    \caption{\small \textbf{Left}: Size of  $\sum_{j=1}^{30} \phi_j(\x_i)$ in $[-1, 1]$. \textbf{Right}: Size of  $\frac{1}{M} \sum_{i=1}^M \left|\sum_{j=1}^N \phi_j(\x_i)\right|$ as $N$ increases with $M=N^2$. }
\label{QQT2}
\end{SCfigure}

\begin{figure}[h]
    \centering
    \includegraphics[width=14cm]{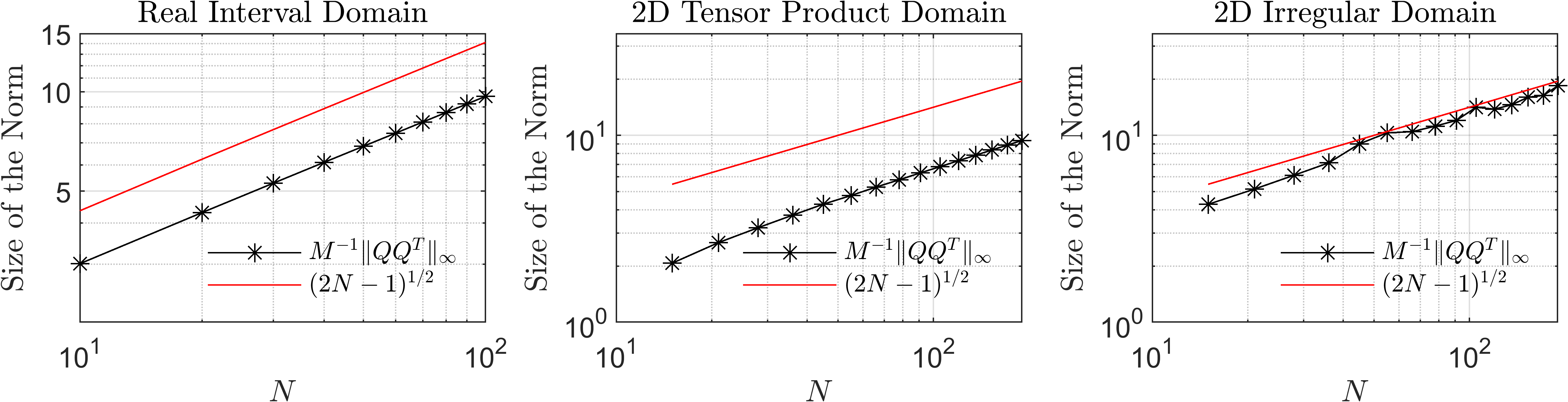}
    \caption{\small Size of  $\frac{1}{M}\|\Q  \Q^*\|_{\infty}$ in different domains. The left two plots are generated with $M = N^2$ equispaced sample points in $[-1, 4]$ and $[-1, 4] \times [-1, 6]$, respectively. The right plot is generated with $M =N^2 \log N$ random sample points on an elliptical domain (Domain $4$ in Figure \ref{domain}). }
    \label{QQT}
\end{figure}
}

\begin{remark}\textbf{(Relationship to Lebesgue Constant)} The bound \eqref{kate bound} in Theorem \ref{thm 4.3} is analogous to the bound in~\cite{MVBleast--squares} involving  the Lebesgue Constant $\Lambda_{\boldsymbol{X}}$, 
\begin{equation}
\tag{Lebesgue Constant Bound}
\|f - \mathcal{L}(f)\|_{\Omega}   \le (1+\Lambda_{\boldsymbol{X}}) \|f- p^*\|_{\Omega}
    \label{Lebesgue Constant bound}
\end{equation}
where the Lebesgue constant $\Lambda_{\boldsymbol{X}}$ is defined as the norm of the linear operator $\mathcal{L}$~\cite{ MVBleast--squares, MVBinterpolation}, 
\begin{align*}
\Lambda_{\boldsymbol{X}} := \min \{\check{c} > 0: \|\mathcal{L}(f)\|_{\Omega}\le \check{c} \|f\|_{\Omega}, \forall f\in \mathcal{C}(\Omega) \}
\end{align*}
where $\mathcal{C}(\Omega) \}$ represents bounded and continuous functions. 
That said, the relationship between \eqref{kate bound} and \eqref{Lebesgue Constant bound} appears to be unknown in the literature. To compare these two bounds, we plot the constant $\frac{2}{M} \|\Q  \Q^*\|_{\infty}$ which is in  \eqref{kate bound} against $\Lambda_{\X}$ which is in \eqref{Lebesgue Constant bound} in Figure  \ref{Leb smooth}. From the plot, we observe that these two bounds are of a similar order of magnitude same order of magnitude in the domains that we tested. 

\sidecaptionvpos{figure}{c}
\begin{SCfigure}[50][ht]
   \centering
    \includegraphics[width=8cm]{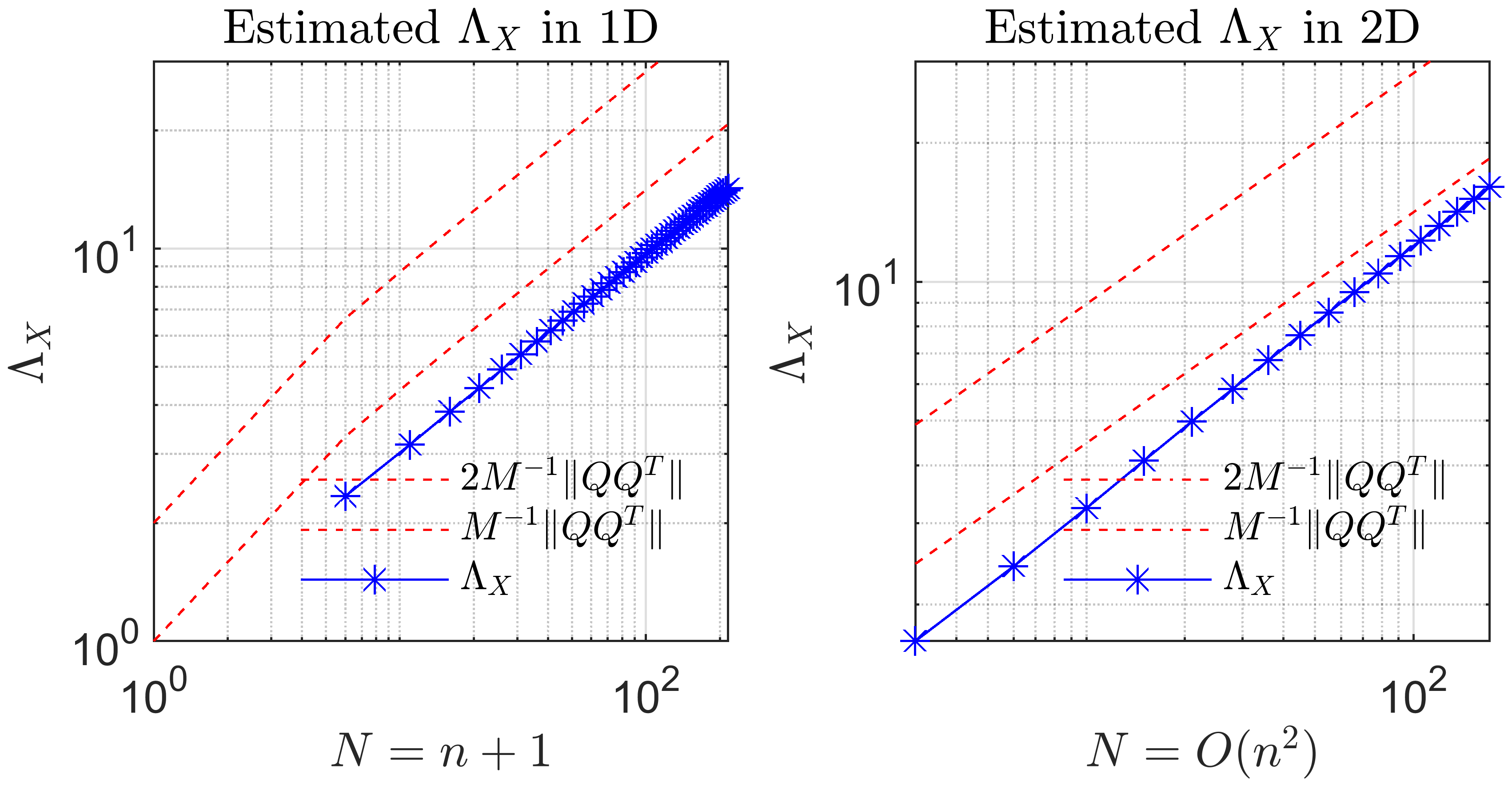}
\caption{\small  {The Lebesgue constant is estimated as $\Lambda_{\boldsymbol{X}} \approx \frac{1}{M} \|\boldsymbol{UQ}^{\ast}\|_{\infty}$, where $\boldsymbol{U} \in \R^{K \times N}$ is the matrix with entries $\{\phi_j(\boldsymbol{y}_i)\}_{1 \le j \le N, 1 \le i \le K} $, and the set of evaluation points $\boldsymbol{Y} := \{\boldsymbol{y}_1, \boldsymbol{y}_2, \dotsc, \boldsymbol{y}_K\} \subset \Omega$ with $K = 10M$. $\Lambda_{\boldsymbol{X}}$ is independent of the approximated function.}}
\label{Leb smooth}
\end{SCfigure}
\end{remark}

\subsubsection{Convergence Rate}
\label{sec Convergence Rate}

Following from Theorem \ref{thm 4.3}, to understand how $\|f-\mathcal{L}(f)\|_{\Omega}$ changes with the {order of the index} $n$, we need to establish a bound for $\|f-p^*\|_{\Omega}$ as $n$ increases. The bound is provided in Theorem \ref{Jackson thm}, adapted from \cite[Thm 1]{bagby2002multivariate}.

\begin{theorem}
\label{Jackson thm}
Let $f$ be a function of compact support on $\R^d$ and $\Omega $ be a compact subset of $\R^d$ which contains the support of $f$. 
{Let $f \in \mathcal{C}^k(\Omega)$ where $ \mathcal{C}^k(\Omega) = \{f: D^{\gamma} f \in \mathcal(\Omega), \forall \gamma: \|\gamma\|_1 \le k\}$ is the set of functions that have continuous derivatives up to total order $k$.}   
Then, 
$$
   \|f-p^*\|_{\Omega} \le C_3 n^{-k}
$$
where $p^*$ is the best polynomial approximation in ${\P^T_{d,n}}$. 
$C_3$ is a positive constant depending only on $k, D_k$ and the diameter of $\Omega $. 
\end{theorem}

\begin{remark}
Theorem \ref{Jackson thm} is the sub-case for \cite[Thm 1]{bagby2002multivariate} where we assume the multi-index for derivative used is $\boldsymbol{\gamma} = 0$. We also assumed that the derivatives of $f$ are bounded, therefore $\omega_{f, k}\big(\frac{1}{n} \big)$ in  \cite[Thm 1]{bagby2002multivariate} is bounded by, 
$\omega_{f, k}\big(\frac{1}{n} \big) := \sup_{{\|\boldsymbol{\gamma} \|_1=k}} \big( \sup_{\|\x-\y\| \le \frac{1}{n}} \|D^{\boldsymbol{\gamma}} f(\x) - D^{\boldsymbol{\gamma}} f(\y)\| \big) \le 2D_k$ where the norms are defined as Euclidean norms.
\end{remark}

Using Theorem \ref{thm 4.3} and \eqref{kate bound} in Theorem \ref{thm 3.3}, we can deduce the rate of convergence of the least squares approximation generated by the V+A algorithm (Algorithm \ref{Multi V+A Algo}).

\begin{theorem}
\label{thm convergence rate equispaced}
 Let $\Omega \subset \R^d$ with $d>1$ be a convex body or finite unions of convex bodies. Let $f:  \Omega \rightarrow {\R^d} \in \mathcal{C}^k(\Omega)$. Construct a set of equispaced sample points as described in Remark \ref{sample method}, ${\boldsymbol{X}} = \big\{ \x_i \big\}_{1\le i \le M} \subset \Omega$ with  $M = \O(N^2)$. Apply Algorithm \ref{Multi V+A Algo} using the sample points $\X$, the error of the least squares approximation satisfies
\begin{eqnarray}
     \|f - \mathcal{L}(f)\|_{\Omega} \le
      \bigg(1+\O(N)\bigg) {O}(n^{-k}) = \O(n^{d-k}) = \O(N^{1-\frac{k}{d}})
     \label{convergence bound}
\end{eqnarray}
where  {$\mathcal{L}(f) \in  \P^T_{d,n}$, $n$ is the order of the indices} and $N$ is the total degree of freedom.
\end{theorem}

\begin{remark} \textbf{(Extension to Random Sample Points)}
With $\O(N^2\log N)$ random sample points chosen as described in Theorem \ref{Ramdon AM}, Theorem \ref{thm 4.3} and Theorem \ref{thm convergence rate equispaced} hold with probability at least $1-M^{-\tilde{c}}$ for all multivariate polynomials in the polynomial space.   
\end{remark}

\begin{remark}\textbf{(Generalization to Domains with Markov Properties)} As we can see in the proof of Theorem \ref{thm 3.3}, the convexity of the domain is only used to deduce \eqref{markov} in Theorem \ref{Theorem markov}. Therefore,  Theorems \ref{thm 3.3}, \ref{Ramdon AM}, \ref{thm 4.3} and \ref{thm convergence rate equispaced} can be generalized to any domains satisfying \eqref{markov}. 
\label{remark Generalization to more domains}
\end{remark}

\begin{example}\textbf{(Convergence Rate for Functions with Different Smoothness)}
We test the least squares approximation using the V+A algorithm for functions of specific smoothness in 1D and 2D. {Note that the index set for the 2D example is the total degree index set.} For smoothness $k = 0,1,2, \infty$ and dimension $d=1,2$, let 
\begin{equation}
 f_{k}(\x ) :=
    \sum_{r=1}^d | x_{(r)}|^{2k+1} \text{ where } \x = (x_{(1)}, \dotsc, x_{(d)})\in [-1,1]^d.
   \label{fnu}
\end{equation}
We also set $f_{\infty}(\x ):= \sum_{r=1}^d  \sin[\exp({x_{(r)}})\cos(x_{(r)})]$. Clearly, $f_{k} \in \mathcal{C}^{k}([-1,1]^d)$ for $k = 0, 1, 2, \infty$.

For fixed $n$, $d$, and $\boldsymbol{X}$, the term $\frac{1}{M}\|\boldsymbol{QQ}^T\|$ is the same for all functions. Thus, the improvement in convergence rates of the smoother functions comes from $\|f-p^*\|_{\Omega}$. Since $\frac{1}{M}\|\boldsymbol{QQ}^T\|_{\infty} \sim \O(\sqrt{N})$ for tensor-product domains,  the least-squares approximations of $f_1$ and $f_2$ converge {like} $\O(n^{-3/2})$ and $\O(n^{-7/2})$ in 1D and $\O(n^{-1})$ and $\O(n^{-3})$ in 2D, respectively. The convergence rates of the four bivariate functions are approximately the square root of the convergence rates of the univariate counterparts. This is expected as we only have $\sim \sqrt{N}$ degrees of freedom in the $x_{(1)}$ or the $x_{(2)}$ directions for the bivariate approximation. 

\sidecaptionvpos{figure}{c}
\begin{SCfigure}[50][ht]
   \centering
    \includegraphics[width=9cm]{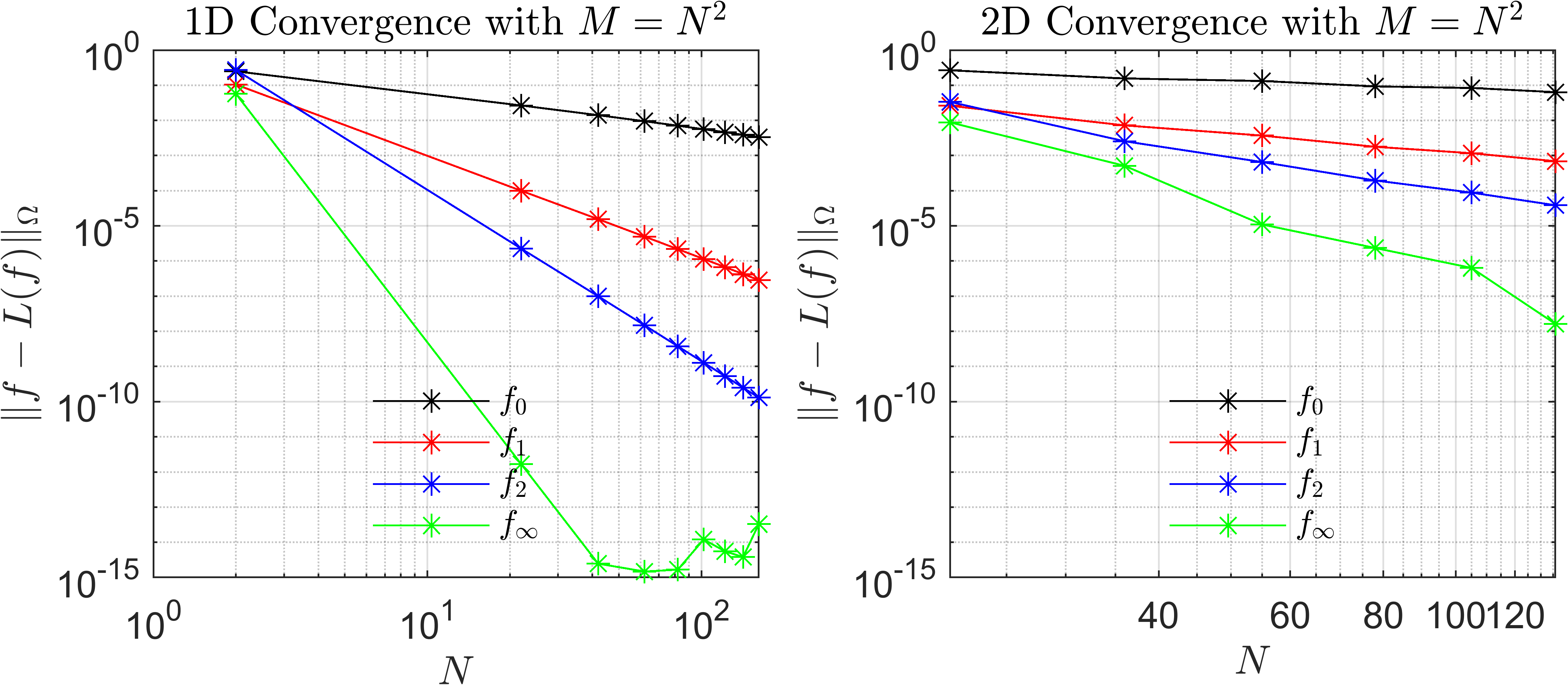}
    \caption{\small The set of sample points $\boldsymbol{X}=\{\x_i\}_{1\le i \le N^2}$ and evaluation points $\boldsymbol{Y}=\{\y_i\}_{1\le i \le 10N^2}$ are chosen as equispaced points. The domain for approximation is $[-1, 1]$ in 1D and $[-1, 1]^2$ in 2D. }
    \label{function smooth}
\end{SCfigure}
\end{example}

\subsection{Related Work on Sample Complexity}
\label{sec Related Work on Sample Complexity}

In this subsection, we discuss the difference in construction between the sample complexity of the least squares approximation using the V+A Algorithm and the sample complexity of polynomial frame approximation.

There are many existing proofs for the sample complexity of polynomial approximations~\cite{frame2, Adcock2019,  cohen, cohenweight,  migdisorth}. 
 For instance, in \cite{cohen, cohenweight}, the proof by Cohen et al. first constructs a $\mathcal{L}^{2}$-continuous orthogonal basis $\{J_1, J_2, \dotsc, J_N\}$, such that $\int_{\Omega} J_i(\x) J_j(\x) dx = \delta_{i, j}. $ Then, the solution of their least-squares problem can be computed by solving the $M \times M$ Gram matrix system, $\Q_1^*\Q_1 d_1 =f$ where the $(i, j)$th entries of $\Q_1$ is $J_j(\x_i)$ for $1 \le j \le N$ and $1 \le i \le M$.  
The purpose of their analysis is to find how many sample points we need such that the discrete measure inherits the orthogonality of $\{J_i\}_{1 \le i \le N}$. Namely, how to choose $M$ such that
\begin{equation*}
   \sum_{i=1}^M{J_j(\x_i)J_k(\x_i)} \approx \delta_{j,k} \qquad\text{and}\qquad \mathbb{E}(\Q_1^*\Q_1) \approx  \boldsymbol{I}_n.
\end{equation*}
Using the exponentially decreasing bounds on tail distributions  $\mathbb{P}(\|\Q_1^*\Q_1 - I\| \ge\frac{1}{2})$,  Cohen et al. proved that $M = \O(N^2 \log N)$ random sample points is enough to obtain a stable least-squares approximation.


In~\cite{Adcock2019}, Adcock and Huybrechs give different proof for the sample complexity of frame approximation on irregular domains. The key step of the proof uses the Nikolskii inequality~\cite{migliorati2015multivariate} on the bounding tensor-product domain
\begin{equation}
    \|p\|_{\Omega_{cube}} \le \mathcal{N}  \|p\|_{\mathcal{L}^{2}(\Omega_{cube})}
\end{equation}
where $\mathcal{N}$ is a constant that depends on the domain and the sample complexity. 
They proved that the least squares are near-optimal if the domain satisfies the $\lambda$-rectangle property. Namely, the bounding domain can be written as a (possibly overlapping and uncountable) union of hyperrectangles $\Omega_{cube}$ where $\lambda$ is the value of $\text{Volume}(\Omega)/ \text{Volume}(\Omega_{cube})$. The sample complexity required scales by $N^2\lambda^{-1}$ where $N$ is the total degree of freedom. {Similar to the Markov constant $\M(\Omega)$ introduced in \eqref{markov}, the parameter $\lambda$ in \cite{Adcock2019} is also independent of the bounding domain and is a constant that depends on $\Omega$. }

Although sample complexity results obtained for the least squares approximation using the frame approximation and using the V+A method are similar, the idea and the construction of the proof are different. Firstly, in polynomial frame approximation, the proof starts with a continuous orthogonal polynomial and explores the distribution of sample points such that the discrete least-squares matrix is approximately orthogonal (i.e., $\Q_1^*  \Q_1\approx I_n$). By construction in V+A, we start with a discrete orthogonal basis (i.e., $\Q^* \Q  = I_n$) and examine the behaviours of these discrete orthogonal polynomials in the domain. Another difference is that these two approaches build the bound using different norm spaces in different domains. In V+A, we use the suprema of discrete orthogonal polynomials, $\|p\|_{\boldsymbol{X}}$, while in the polynomial frame approach, the Nikolskii inequality uses the $\mathcal{L}_2$ norm, $\|p\|_{\mathcal{L}^2(\Omega_{cube})}$. 
{The analysis for the least squares approximation using V+A method applies to the maximum degree or total degree polynomial spaces. In contrast, the analysis in \cite{Adcock2019} provides a more general result regarding polynomial spaces, extending to any polynomial space defined by a lower set. {This includes maximum degree, total degree, and hyperbolic cross polynomial spaces. Current sample complexity analysis (derived via Lemma \eqref{lemma markov}) results in a scaling of $M = \O\left(\frac{N}{\log^{d-1}(N)}\right)^{2d}$ for the hyperbolic cross index set. This scaling is less favorable compared to the $M = \O(N^2)$ scaling reported in \cite{Adcock2019}, although the domain assumptions differ, and the error bound in that case is measured in the $L_2$ norm rather than the $L_\infty$ norm. Further investigation into improving the sample complexity analysis for hyperbolic cross polynomial spaces remains an open area of research.}
In terms of the domain, the analysis in \cite{Adcock2019} requires the bounding domain can be written as a (possibly overlapping and uncountable) union of hyperrectangles $\Omega_{cube}$, whereas the analysis for the least squares approximation using V+A method holds for all domain satisfying \eqref{markov} (Remark \ref{remark Generalization to more domains}). }

More recently, Adcock and Shadrin \cite{adcock2023fast} demonstrated that with linear oversampling of equispaced sample points and polynomial extensions over an extended interval, the equispaced samples become near-optimal for polynomial frame approximation in 1D. They propose a well-conditioned method that ensures exponential decay of the error, reaching a finite, user-controlled tolerance.
In the following section, we will explore the use of the weighted least-squares algorithm and the multivariate V+A algorithm to give a well-conditioned and near-optimal polynomial approximation. 


\section{Near-Optimal Sampling Strategy for V+A}
\label{section Weighted LSA}

In this section, we propose a new variant of the weighted least-squares algorithm that uses the multivariate V+A to create the discrete orthogonal basis. We refer to our V+A weighted least-squares algorithm as \texttt{VA+Weight}.  

In~\cite{ADCOCK, adcock2022towards, cohenweight, migdisorth}, the authors gave comprehensive analyses on this weighted sampling strategy and proved that only $M=\O(N\log N)$ sample points are needed for a well-conditioned and accurate approximation. {A recent work~\cite{adcock2024optimal} provides a comprehensive review on optimal sampling for (weighted) least-squares approximation in arbitrary linear spaces, introducing the Christoffel function as a key quantity in the analysis of (weighted) least-squares approximation from random samples.} In \texttt{VA+Weight}, we use the same weighting measure as in~\cite{ADCOCK}.
But instead of creating the discrete orthogonal basis with QR factorization, we use V+A as the orthogonalization strategy for the Vandermonde basis. Since the Vandermonde matrix $\A $ is usually highly ill-conditioned, and computing its $Q$ factor incurs errors proportional to the condition number of $\A$~\cite[Ch.~19]{Thm19.4}. Even though the QR factorization is a stable orthogonalization technique, the discrete orthogonal basis generated from $\A $ could still be inaccurate. We refer to the weighted least-squares approximation using the QR factorization as \texttt{QR+Weight}. 
Along with multiple numerical examples, we illustrate that
\texttt{VA+Weight} gives more accurate approximations than \texttt{QR+Weight} for high-degree polynomial approximations.  Due to the reduced sample density, \texttt{VA+Weight} also gives a lower online computational cost than the unweighted V+A least-squares method. Moreover, \texttt{VA+Weight} acts as a practical tool for selecting the near-optimal distribution of sample points in a high-dimensional irregular domain. 

The section is arranged as follows. In Section \ref{Section setup},  we explain the algorithm and the numerical setup for the weighted least-squares approximation. We provide proof of the stability of the weighting measure for sample points. In Section \ref{Section Numerical VA Weight}, we give numerical examples to compare the \texttt{VA+Weight} and \texttt{QR+Weight} algorithms.

\subsection{Weighted Least-Squares Approximation}
\label{Section setup}

The \texttt{VA+Weight} algorithm is presented in Algorithm \ref{Algo weighted least-square} and the remarks for the algorithm are given in Remark \ref{remark1-3}. Algorithm \ref{Algo weighted least-square} is a V+A variant of Method 1 in~\cite{ADCOCK}.

\begin{algorithm}[!ht]\small
\textbf{Input: }A compact and bounded domain $\Omega$ {satisfying \eqref{markov} and the dimension of the domain $d$.} Bounded and continuous $f \in \mathcal{C}(\Omega)$; 
\\${\P^T_{d,n}}$ with {Vandermonde basis} $\hat{\psi} =[\hat{\psi}_1(\x ), \dotsc, \hat{\psi}_N(\x )]^T$; 
\\Number of sample points {$M = \O(N^2\log N)$ and $\hat{M}  = \O(N\log N)$ such that $M\ge \hat{M} \ge N$.}
\\ \textbf{Output}: The coefficients $\hat{d}\in \R^{N}$ of the polynomial approximant. 
\smallbreak
\textbf{Step 1:} Draw $M$ random sample points $\boldsymbol{X} = \{\x_i\}_{1\le i \le M} \stackrel{i.i.d.}{\sim} \rho$ {where $\rho$ represents the uniform measure}. Compute the function values at $\boldsymbol{X}$ i.e $\boldsymbol{\tilde{f}} \in  \R^M$.
\\
\textbf{Step 2:} Construct $M \times N$ Vandermonde matrix $\A $ with the $(i, j)$th entry as $\hat{\psi}_j(\boldsymbol{y}_i)$. If $\rank(\A ) =N$, go to Step 3, else go back to Step 1. 
\\
\textbf{Step 3:} Apply Algorithm \ref{Multi V+A Algo} to $\A $ to generate $\Q  \in \R^{M \times N}, \boldsymbol{H} \in \R^{N \times N}$ such that $\mbox{diag}(\boldsymbol{X})\Q =\Q \boldsymbol{H}+h_{N+1,N}q_{N+1}e_{N}$.
\\
\textbf{Step 4:} Define a probability distribution $\pi = \{\pi_i\}_{1\le i \le M}$ on $\{1, \dotsc, M\}$, such that
$
    \pi_i = \frac{1}{\|\Q \|_{F}^2} \sum_{j=1}^N |\Q_{i,j}|^2,$ for $i=1, \dotsc, M,
$
where {$\|\Q \|_F :=  \sum_{j=1}^M\sum_{j=1}^N |\Q_{i,j}|^2$ is the Frobenius norm} and $\Q_{i,j}$ are the $(i, j)$th entry of $\Q$. 
\\
\textbf{Step 5:} Draw $\hat{M}$ integers $\{k_1, \dotsc, k_{\hat{M}}\}$ independently from $\pi$. Define $\hat{\Q }$ and $\boldsymbol{\hat{f}}$ as the corresponding scaled rows of $\Q $ and $\boldsymbol{\tilde{f}}$, such that the point-wise entries are
\begin{equation*}
    \hat{\Q }_{i, j}:= \frac{\Q_{k_i,j}}{\sqrt{\hat{M}M\pi_{k_i}}},  \qquad  \boldsymbol{\hat{f}}_{i}:= \frac{f(\x_{k_i})}{\sqrt{\hat{M}M\pi_{k_i}}}, \qquad i=1, \dotsc, \hat{M}, \quad j=1, \dotsc, N. 
\end{equation*}
\\
\textbf{Step 6:} $\boldsymbol{\hat{d}}= \boldsymbol{\hat{Q}} \backslash \boldsymbol{\hat{f}}. $ Solve the least-squares problem by the MATLAB backslash command. Approximate the value of $f$ using the evaluation algorithm in Algorithm \ref{Multi eval Algo} and give the output. 
\caption{\small The Weighted Least Squares Approximation with V+A Method (\texttt{VA+Weight})}
\label{Algo weighted least-square}
\end{algorithm}

\begin{remark}
Details for implementing {Algorithm \ref{Algo weighted least-square}} are as follows. 
\begin{enumerate}
\item  {$\Omega \in \R^d$ could be a convex body or finite unions of convex bodies or other domains that satisfy \eqref{markov} (See Remark \ref{remark Generalization to more domains})}
\item {In our numerical implementation, unless otherwise stated,  we choose $M = N^2\log N$ and $\hat{M}  = N\log N$ such that $M\ge \hat{M} \ge N$}.
\item  We assume that it is possible to draw samples from the {uniform} measure $\rho$ in Step 1.  We use the uniform rejection sampling method to draw samples. We end the rejection sampler once we have enough sample points. 
\item  To ensure that $\spn\{\phi_1, \dotsc,\phi_N\}={\P^T_{d,n}}$  
in Step 2 of the algorithm, if $\A $ is rank {deficient}, 
we add additional sample points until $\rank(\A ) = N$. 

\item The construction of $\Q $ uses Algorithm \ref{Multi V+A Algo}. Numerically, the orthogonalization algorithm is subject to a loss of orthogonality due to numerical cancellation. The numerically constructed discrete orthogonal basis is said to be $\epsilon_{m}$-orthonormal for $\epsilon_{m}>0$, namely, $\|\Q^*\Q -\boldsymbol{I}\|^2_F=\sum_{j,k=1}^N|\langle \phi_j, \phi_k\rangle_{M} - \delta_{j, k}|^2 \le \epsilon_{m}^2$ where $\langle  \phi_j, \phi_k \rangle_M =  \frac{1}{M}\sum_{i=1}^M \phi_j(\x_i) \phi_k(\x_i)$.   For any bounded domain $\Omega$, we define \textit{the infinity $\Omega$-norm} of a bounded function $g: \Omega \rightarrow \R$ as $\|g\|_{\Omega}=\sup_{\x \in\Omega}|g(\x )|$. 
In the multivariate V+A algorithm, we execute CGS twice for the orthogonalization, which gives a bound $\epsilon_m \sim \O(MN^{3/2}) \boldsymbol{u}$, where $\boldsymbol{u}$ is the unit roundoff~\cite[Thm 19.4]{Thm19.4} and \cite[Thm 2]{gstwice}. 
\end{enumerate}
\label{remark1-3}
\end{remark}

In Algorithm \ref{Algo weighted least-square}, for each sample point $\x_i$, $\sum_{j=1}^N |\Q_{i,j}|^2$ represents the sum of the absolute value of discrete orthogonal polynomials at $\x_i$. The weighting measure $\pi$ can be interpreted as choosing the sample points that maximize the absolute sum of the discrete orthogonal polynomials at the sample points. Heuristically, this weighting measure makes sense as the supremum usually happens near the boundaries and corners of the domain. Many sampling measures, such as Chebyshev points in real intervals and Padua points~\cite{bos2006bivariate, caliari2005bivariate} in higher dimensional tensor-product domains have highlighted the importance of sample points near the boundary and corners. 

$\hat{M}$ is the number of points we selected from a total of $M$ sample points. The main question that we analyze in this section is how large $\hat{M}$ needs to be chosen in relation to $N$ to ensure a near-optimal approximation. As we will prove later, with the probability distribution defined in Step 4, a log-linear scaling of $\hat{M} = \O(N\log N)$ is enough for a well-conditioned, near-optimal weighted least-squares approximation (provided that $M$ is large enough). The reduced sample complexity in  Algorithm \ref{Algo weighted least-square} gives a reduced online computational cost. The computational cost for Algorithm \ref{Algo weighted least-square} is dominated by the cost of Step 3 and Step 6, which are $\O(MN^2)$ flops and $\O(\hat{M}N^2)$ flops, respectively. Since we need $M = \O(N^2\log N)$ random sample points to generate a randomized admissible mesh for convex domains or unions of convex domains in $\R^d$, Algorithm \ref{Algo weighted least-square} gives an online computational cost of $\O(N^3\log N)$ while the unweighted least-squares method described in Section \ref{section Multivariate VA} requires $\O(N^4\log N)$.

Algorithm \ref{Algo weighted least-square} can be interpreted as a weighted least-squares system with V+A orthogonalization.
Using the weight matrix 
$
    \boldsymbol{\hat{W}} :
= \frac{1}{\sqrt{\hat{M}}}\diag\left(\frac{1}{\sqrt{M \pi_{k_i}}}\right)_{1 \le i \le \hat{M}},
$
the weighted least-squares problem can be written as
\begin{eqnarray} 
\boldsymbol{\hat{d}} = \argmin_{\boldsymbol{{d}} \in \R^N} \|\boldsymbol{\hat{Q}d}-\boldsymbol{\hat{f}}\|_{2}:= \argmin_{\boldsymbol{{d}} \in \R^N} \|\boldsymbol{\hat{W}{Q_S}d}-\boldsymbol{{f_S}}\|_{2} 
    \label{weighted ls}
\end{eqnarray} where 
$\boldsymbol{Q_S} \in \R^{\hat{M}\times N}$ and $\boldsymbol{f_S} \in \R^{\hat{M}}$ are the matrix and the vector formed with $\{k_1, \dots, k_{\hat{M}}\}$ selected rows of $\Q $ and $\boldsymbol{\tilde{f}}$ in the unweighted system. $\boldsymbol{\hat{d}}: =[\hat{d}_1, \dotsc, \hat{d}_N]^T$ is the vector of coefficients for the weighted least-squares estimator such that $\mathcal{L}(f)_W:= \sum_{i=1}^N \hat{d}_i \phi_i$. Note that $\phi_i$ is defined as before, namely the discrete orthogonal polynomials generated by $M$ sample points. The estimator $ \mathcal{L}(f)_W$ is solved using normal equations, such that
\begin{equation}
 \boldsymbol{d} = \boldsymbol{G}^{-1} (\boldsymbol{\hat{Q}^*\hat{f}}). 
 \label{gram}
\end{equation}
$\boldsymbol{G}:=\boldsymbol{\hat{Q}^*\hat{Q}} \in \R^{N\times N}$ is the reduced Gram matrix. 

Since the weighted least-squares estimator is found by solving system \eqref{gram} and by taking the inverse of the matrix $\boldsymbol{G}$, {for stability and convergence} we need to ensure that the Gram matrix $\boldsymbol{G}$ is well-conditioned. Also, we investigate how much  $\boldsymbol{G}=\boldsymbol{\hat{Q}^*\hat{Q}} $ deviates from $\boldsymbol{{Q}^*{Q}}$. In other words, we want to understand whether the discrete orthogonal basis at the selected sample points $\boldsymbol{\hat{X}}$ acts as a good approximation to the discrete orthogonal basis at the full sample points $\boldsymbol{{X}}$. The following theorem adapted from \cite[Thm 3]{migdisorth} establishes these links. 

\begin{theorem}\textbf{(Well-Conditioning of Reduced Gram Matrix)}
Let $\Omega \in \R^d$, consider finding a weighted least-squares approximation in the polynomial space ${\P^T_{d,n}}$ using Algorithm \ref{Algo weighted least-square}. Let $\boldsymbol{X} = \{\x_i\}_{1\le i \le M} \stackrel{i.i.d.}{\sim} \rho$ and generate the discrete orthogonal basis using Step 3 of Algorithm \ref{Algo weighted least-square}. The numerically constructed discrete orthogonal basis is $\epsilon_{m}$-orthonormal for $\epsilon_{m}>0$, such that $\|\Q^*\Q -\boldsymbol{I}\|^2_F=\sum_{j,k=1}^N|\langle \phi_j, \phi_k\rangle_{M} - \delta_{j, k}|^2 \le \epsilon_{m}^2$.  For $\hat{\alpha} \in (0, \frac{1}{2}), \epsilon_m \in (0,1), \hat{\delta} \in (0,1-\epsilon_m)$,  and $n \ge 1$, if the following conditions hold,
\begin{flalign*}
\text{i) } &\hat{M}\ge \frac{4N(1+\epsilon_m)}{\hat{\delta}^2} \log (\frac{2N}{\hat{\alpha}}), &&\\
\text{ii) }&\boldsymbol{\hat{X}} = \{\boldsymbol{\hat{x}}_i\}_{1\le i \le \hat{M}} \stackrel{i.i.d.}{\sim} \pi \text{ where $\pi$ is defined as in Step 4 of Algorithm \ref{Algo weighted least-square},}&& 
\end{flalign*}
then, the matrix $\boldsymbol{G}$ satisfies $\mathbb{P} (\|\boldsymbol{G}-\boldsymbol{I}\|_2\ge \hat{\delta}+\epsilon_m) \le \hat{\alpha}$ where $\boldsymbol{I}$ is the $N \times N $ identity matrix. 
\label{Thm weighted convergence}
\end{theorem}
We give a sketch of the proof for Theorem \ref{Thm weighted convergence}. We write $\mathbb{P} (\|\boldsymbol{G}-\boldsymbol{I}\|_2\ge \hat{\delta}+\epsilon_m)$ as
$$\mathbb{P} (\|\boldsymbol{G}-\boldsymbol{I}\|_2 < \hat{\delta}+\epsilon_m) \ge \underbrace{ \mathbb{P}(\{\|\mathbb{E}(\boldsymbol{G})-\boldsymbol{I}\|_2 < \epsilon_m  \})}_{:=P_1} \underbrace{\mathbb{P} ( \{\|\boldsymbol{G}-\mathbb{E}(\boldsymbol{G})\|_2 < \hat{\delta}\})}_{:=P_2}.$$
The first probability term $P_1$ is bounded using condition $(ii)$, the equality $\mathbb{E}(\boldsymbol{G}_{j,k}) =  \langle \phi_j, \phi_k \rangle_{M}$ and the $\epsilon_m$-orthogonality. The second probability term $P_2$ is bounded using the Bernstein Inequality and condition $(i)$ which gives a tail bound for sums of random matrices. A full proof of Theorem \ref{Thm weighted convergence} can be found in~\cite[Thm 3]{migdisorth}. 

Theorem \ref{Thm weighted convergence} not only ensures that $\boldsymbol{G}$ is well-conditioned with high probability but also guarantees that the weighted least-squares problem \eqref{gram} is stable with high probability. Under the conditions of Theorem \ref{Thm weighted convergence}, we have $1-\hat{\delta} -\epsilon_m \le \|\boldsymbol{G}\|_2 \le 1+\hat{\delta} +\epsilon_m$. Using that $\boldsymbol{v}^T\boldsymbol{G}\boldsymbol{v} = (\boldsymbol{\hat{Q}v})^T(\boldsymbol{\hat{Q}v})$ for all $\boldsymbol{v}\in \R^N$, it follows that
\begin{eqnarray}
    \|\boldsymbol{\hat{Q}}\|_2 =
    \|\boldsymbol{G}\|_2^{1/2}, \text{ and } \|\boldsymbol{G}^{-1}\|_2\|\boldsymbol{\hat{Q}}^T\|_2 \le \frac{\sqrt{1+\hat{\delta} +\epsilon_m}}{1-\hat{\delta} -\epsilon_m} =: C_{\hat{\delta}, \epsilon_m}.
\label{QG bound}
\end{eqnarray}
Thus, we arrive at the stability result $\|\boldsymbol{\hat{d}}\|_2 =  \|\boldsymbol{G}^{-1}\boldsymbol{\hat{Q}}\boldsymbol{\hat{f}}\|_2 \le  C_{\hat{\delta}, \epsilon_m} \|\boldsymbol{\hat{f}}\|_2.$ Numerical examples the change of condition number of $\boldsymbol{G}$ with respect to $N$ can be found in Figure \ref{condition of gram}.  {We illustrate, with a numerical example, the growth of the condition number of Gram Matrix (i.e., $\kappa_2(\boldsymbol{G})$) with respect to $N$ for different sampling densities $\hat{M}$. In Figure \ref{mathcalC}, we plot the condition number of $\boldsymbol{G}$ formed by the weighted sampling method in Step 5 of Algorithm \ref{Algo weighted least-square}. For each line in Figure \ref{mathcalC}, we choose $\hat{M}$ to be different functions of $N$. As illustrated by the plot, 
we need a sampling density of $\hat{M}=\O(N\log N)$ to prevent 
$\kappa_2(\boldsymbol{\hat{Q}})$ from growing exponentially.} 

\sidecaptionvpos{figure}{c}
\label{condition of gram}
\begin{SCfigure}[50][ht]
   \centering
    \includegraphics[width=9cm]{mathcal_C.jpg}
    \caption{\small The sample points in this plot are generated from Domain $2$ in Figure \ref{domain}. }
    \label{mathcalC}
\end{SCfigure}

To ensure the convergence of Algorithm \ref{Algo weighted least-square}, in addition to conditions $(i)$ and $(ii)$ in Theorem \ref{Thm weighted convergence}, we also require a sample density of $M=\O(N^2\log N)$ random sample points. This result is expected as we do need $M=\O(N^2\log N)$ random sample points to form discrete orthogonal polynomials that are well-bounded in the domain. There are a few papers that discuss the convergence for Algorithm \ref{Algo weighted least-square}, namely \cite[Thm 3.1-3.5]{ADCOCK}, \cite[Thm 6.6]{Adcock2019} and \cite[Thm 2]{migdisorth}. 

\subsection{Numerical Examples for Weighted V+A Algorithm}
\label{Section Numerical VA Weight}
In this subsection, we compare the \texttt{VA+Weight} algorithm with the \texttt{QR+Weight} algorithm. The two algorithms use the same weights, but the \texttt{QR+Weight} algorithm uses the QR {factorization of the ill-conditioned Vandermonde matrix} to create the discrete orthogonal basis. 

In Figure \ref{Sampling 1D}, we plot the numerical results of approximating a smooth function using \texttt{VA+Weight} and \texttt{QR+Weight} in a real interval. As shown in the left plot, both algorithms converge with $\hat{M}=\O(N\log N)$ number of weighted sample points. Before the ill-conditioning of $\A $ builds in, the two algorithms performed similarly as expected. Note that the error from {the two} approximations will not be the same as we selected the weighted sample points in a non-deterministic fashion (i.e., following the probability measure $\sigma$). However, when the condition number of $\A$ grows beyond the inverse of the machine epsilon for $N>60$, the \texttt{QR+Weight} approximation has an error stagnating at $10^{-5}$ as illustrated in the left plot of Figure \ref{Sampling 1D}. Although the QR factorization and the weighted sampling method {generate} a well-conditioned $\boldsymbol{\hat{Q}}$  
 (Figure \ref{Sampling 1D}), the columns of $\boldsymbol{\hat{Q}}$ do not approximate the orthogonal basis in the domain (Figure \ref{discrete orth qr}). Thus, accuracy is lost for large $N$.

\begin{figure}[!ht]
    \centering
    \includegraphics[width=14cm]{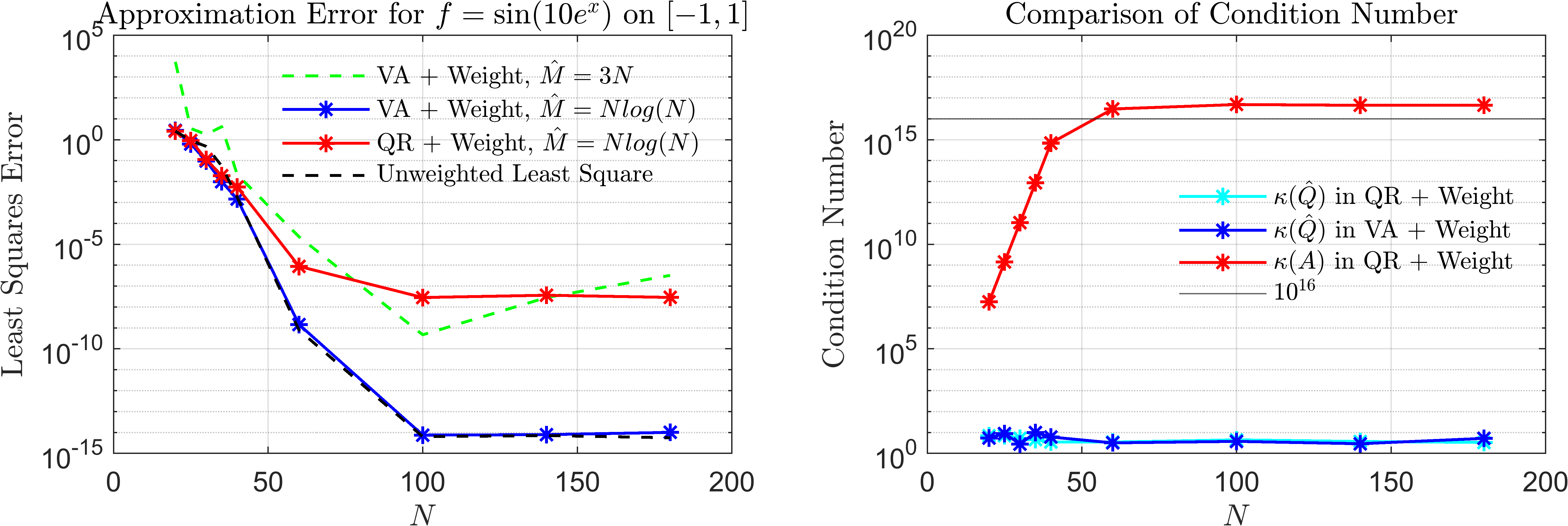}
    \caption{\small Comparison of \texttt{VA+Weight} and \texttt{QR+Weight} in a 1D domain. The sample points are chosen as $M=N^2 \log N$ random points.  The unweighted least-squares approximation is computed by Algorithm  \ref{Multi V+A Algo}. }
\label{Sampling 1D}
\end{figure}

On the other hand, as shown in the right plot of Figure \ref{Sampling 1D}, \texttt{VA+Weight} approximation gives a stable error reduction down to $10^{-13}$. 
The error reduction in \texttt{VA+Weight} also matches with the error reduction in the unweighted least-squares approximations. This is because the discrete orthogonal polynomials generated by V+A are unaffected by the ill-conditioning of $\A $.  The value of the $30$th discrete orthogonal polynomial overlaps with the value of the $30$th scaled Legendre polynomial in the domain (Figure \ref{discrete orth qr}).

\begin{figure}[!ht]
  \centering
        \includegraphics[width=9cm]{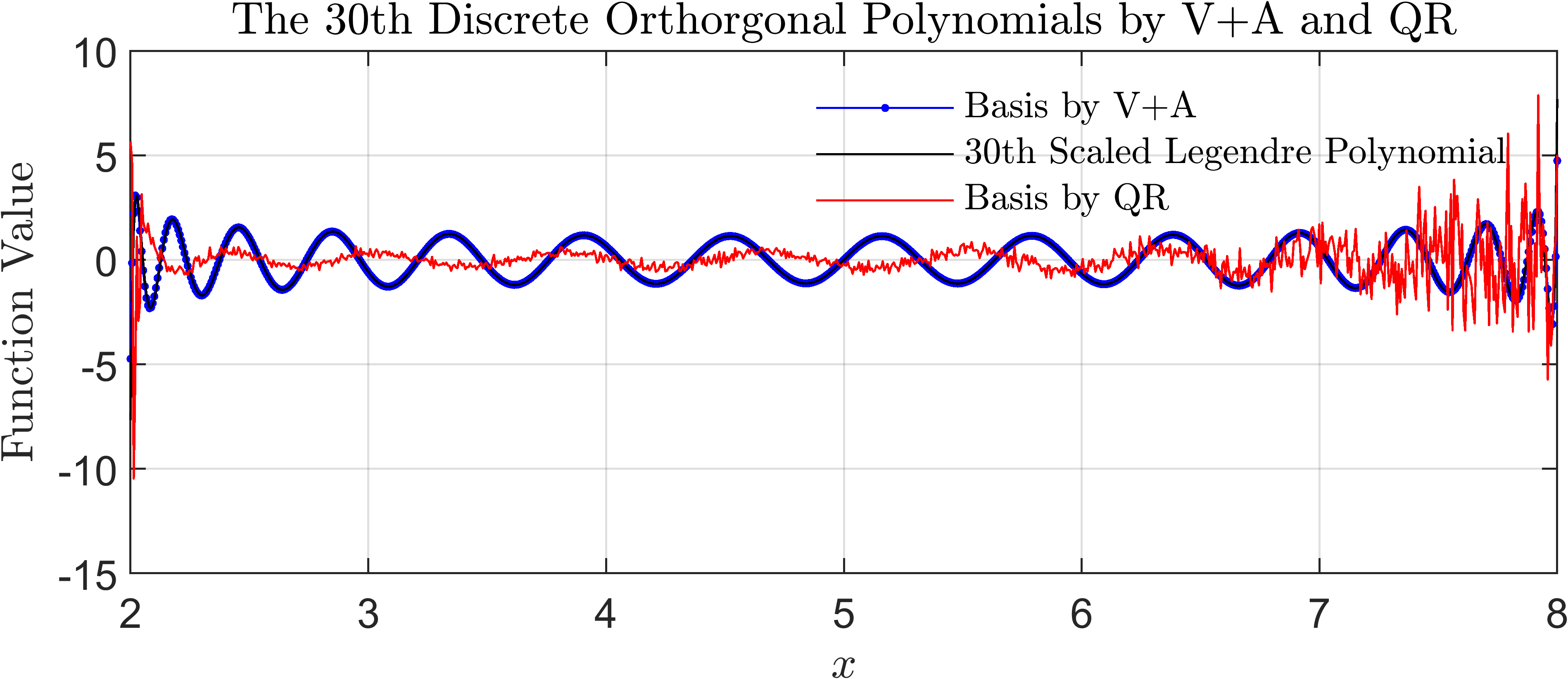}
    \caption{\small The discrete orthogonal polynomials and the scaled Legendre polynomials generated by \texttt{VA+ Weight} and  \texttt{QR+ Weight} using $M = 900$ equispaced points in $[-5, 10]$.}
    \label{discrete orth qr}
\end{figure}

The same patterns of convergence are found while approximating bivariate functions. As plotted in Figure \ref{Sampling 2D}, \texttt{VA+Weight} gives an approximation with higher accuracy than  \texttt{QR+Weight} in both domains. The difference in the two approximations is less in the 2D domain than in the 1D domain. This is because the multivariate Vandermonde matrix $\A $ is, in general, less ill-conditioned in 2D domains than in 1D domains. That said, the conditioning of the multivariate Vandermonde matrix $\A $ varies greatly with the shape of the domain. \texttt{VA+Weight} provides a stable and generalized method for multivariate approximations in irregular domains. 

\begin{figure}[!ht]
    \centering
    \includegraphics[width=14cm]{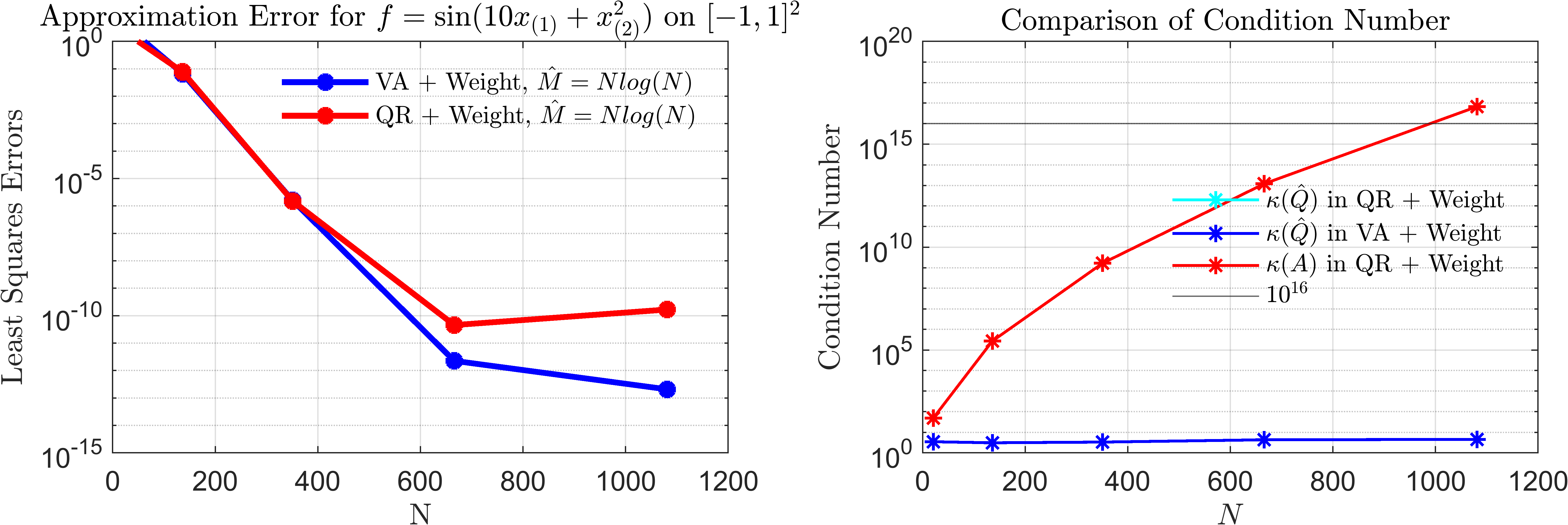}
    \includegraphics[width=14cm]{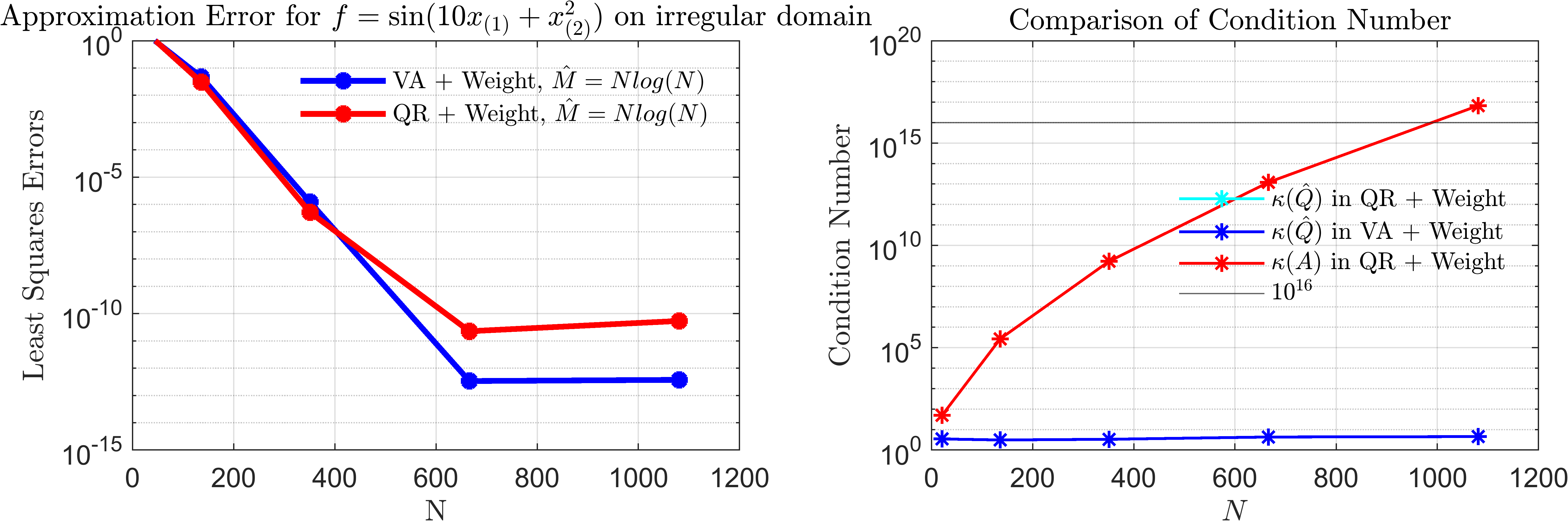}
    \caption{\small Comparison of \texttt{VA+Weight} and \texttt{QR+Weight} in 2D domains. The sample points are chosen as $M=N^2 \log N$ random points. We use Domain $2$ in Figure \ref{domain}.  }
    \label{Sampling 2D}
\end{figure}

Finding the best distribution of sample points in a high-dimensional irregular domain is an open question in the literature. We highlight that the weighting method in Algorithm \ref{Algo weighted least-square} is a practical tool for finding the near-optimal distribution of the sample points. We plot the $\hat{M}=N\log N$ sample points selected by \texttt{VA+Weight} for different domains in Figure \ref{domain}. Notice that the selected points are clustered near corners and boundaries.  This distribution pattern matches the pattern in Padua points in tensor-product domains. {This behaviour occurs because the measure we use to select sample points assigns a higher probability to points where the discrete orthogonal polynomials have larger absolute values. Typically, these points cluster near the boundaries. A simple example is the set of Chebyshev polynomials or Legendre polynomials in the interval $[-1, 1]$, where the polynomials exhibit larger absolute values at the endpoints.}

\begin{figure}[!ht]
   \centering
        \includegraphics[width=14cm]{domain.jpg}
        \caption{\small $M=N^2\log N$ random sample points are drawn from the domain with $N=200$, plotted in the yellow dots. The weighted sample points selected are plotted in blue dots. }
    \label{domain}
\end{figure}

Finally, we note that different weighting measures can be used to improve different aspects of the approximation algorithm. As illustrated in the left two plots of Figure \ref{best 2D3}, the \texttt{VA+Weight} algorithm gives a similar accuracy as the unweighted V+A least-squares approximation but with improved efficiency. That said, the approximants obtained from \texttt{VA+Weight} are only near-optimal, but not the best polynomial approximation.  We propose combining the multivariate V+A with Lawson's algorithm (\texttt{VA+Lawson}) to obtain the best multivariate polynomial approximation. \texttt{VA+Lawson} is based on an iterative re-weighted least-squares process and can be used to improve the approximation accuracy. Our numerical experiments showed that through the \texttt{VA+Lawson} algorithm, we improved the approximation error from $1.4\times 10^{-6}$ to $3.2\times 10^{-7}$ . Moreover, we found an equioscillating error curve in 2D, a characteristic of the best polynomial approximation~\cite[Thm. 24.1]{ATAP}.

\begin{figure}[!ht]
    \centering
    \includegraphics[width=14cm]{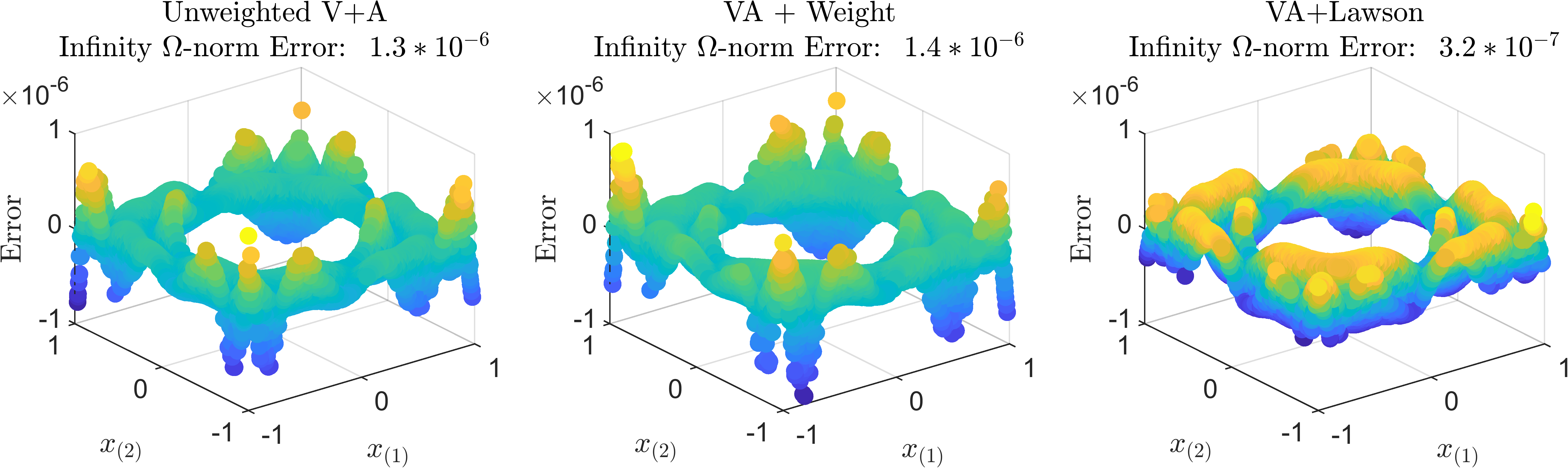} 
    \caption{\small Approximating $f= \sin(x_{(1)}x_{(2)})$ with total degrees of freedom $N=66$ with $M=N^2\log N$ random sample points. We use the same domain as the middle plot of Figure~\ref{domain}. The error curve of the approximant using V+A, \texttt{VA+Weight}, \texttt{VA+Lawson} are plotted in the left, middle, and right plots, respectively. For \texttt{VA+Lawson}, we carried out $10$ Lawson's iterations. }
    \label{best 2D3}
\end{figure}

\section{Conclusions and Future Works}

In this paper, we analyzed the multivariate Vandermonde with the Arnoldi method to approximate $d$-dimensional functions on irregular domains. The V+A technique resolves the ill-conditioning issue of the Vandermonde matrix and builds the discrete orthogonal polynomials with respect to the domain. Our main theoretical result is the convergence of the multivariate V+A least-squares approximation for a large class of domains. The sample complexity required for convergence is quadratic in the degree of freedom of the approximation space, up to a logarithmic factor {polynomial approximations in total degree and maximum degree polynomial spaces.} {For hyperbolic cross polynomial spaces, \cite{Adcock2019}  proved a sample complexity scaling of $M = \O(N^2)$ under the $L_2$ norm, the study of hyperbolic cross polynomial spaces under the $L_\infty$ norm remains a topic for our future research.}
Using a suitable weighting measure, we showed that the sample complexity can be further improved. The weighted V+A least-squares method requires only log-linear sample complexity $M = \O(N \log(N))$. 

{Our preliminary numerical results validate that the least squares approximation using the V+A method provides well-conditioned and near-optimal approximations for multivariate functions on (irregular) domains with the dimension $d$ ranging from $1$ to $5$. Additionally, the (weighted) least squares approximation using the V+A method performs competitively with state-of-the-art orthogonalization techniques and can serve as a practical tool for selecting near-optimal distributions of sample points in irregular domains.
It is important to note that Algorithm \ref{Multi V+A Algo} and the theories we prove in this paper for multivariate V+A are applicable for $d \ge 2$. However, due to the growth of the polynomial basis with the dimensionality of functional approximation {(i.e., $N = \O(n^d)$ for total degree and maximum degree polynomial spaces)}, the CPU time required for computations for even higher dimensions will soon become infeasible due to the curse of dimensionality. 
The development of practical and scalable algorithms for high-dimensional functional approximation is also deferred to future work.

V+A has many applications beyond least-squares fitting. Our numerical experiments showed that the \texttt{VA+Lawson} algorithm improves the approximation accuracy and generates an equioscillating error curve. Yet, the convergence profile for \texttt{VA+Lawson} still seems to be unknown and could be an objective of future work. Another extension under consideration involves using multivariate V+A  in vector- and matrix-valued rational approximations~\cite[Subsec. 2.4]{skiter}. Generally speaking, in any application that involves matrix operations of Vandermonde matrix (or its related form), it seems likely that the V+A procedure could be an effective idea to apply. 

\textbf{Acknowledgments}: We thank the referees for their valuable input, which has improved our work. We also appreciate the questions and comments from Ben Adcock, which led to our research on the convergence of V+A in {different polynomial spaces and smoothness spaces.}

\clearpage
\newpage
\appendix

\clearpage

\scriptsize{
\bibliographystyle{plain}
\bibliography{bib_file.bib}
}

\end{document}